\documentclass[11pt]{amsart}
\usepackage{amsmath}
\usepackage{amsthm}
\usepackage{amsfonts}
\usepackage{color}
\usepackage{soul}
\usepackage{amssymb}
\usepackage{amscd}
\numberwithin{equation}{section}
\newtheorem{theorem}{Theorem}[section]
\newtheorem{cor}[theorem]{Corollary}
\newtheorem{proposition}[theorem]{Proposition}
\newtheorem{lemma}[theorem]{Lemma}

\theoremstyle{definition}
\newtheorem{defi}{Definition}[section]
\newtheorem{rem}[defi]{Remark}

\newtheorem{example}[defi]{Example}

\newcommand\half{\frac{1}{2}}

\newcommand\ov{\overline}

\newcommand\g{\mathfrak g}

\newcommand\h{\mathfrak h}
\newcommand\ha{\widehat{\mathfrak h}}

\newcommand\D{\Delta}
\renewcommand\l{\lambda}

\renewcommand\d{\delta}

\renewcommand\a{\alpha}
\renewcommand\aa{\mathfrak a}

\newcommand{\Z}{\mathbb Z}
\renewcommand\th{\theta}

\newcommand\nat{\mathbb N}
\newcommand\ganz{\mathbb Z}

\renewcommand\L{\Lambda}

\renewcommand\aa{\mathfrak a}

\newcommand\e{\epsilon}
\newcommand\C{\mathbb C}
\newcommand\R{\mathbb R}

\renewcommand\ha{\widehat{\mathfrak h}}

\newcommand{\ZZ}{\mathbb{Z}}

\newcommand{\Dim}{\text{\rm Dim}}
\newcommand{\sdim}{\text{\rm sdim}}

\newcommand{\vac}{{\bf 1}}
\newcommand{\bea}{\begin{eqnarray}}
\newcommand{\eea}{\end{eqnarray}}
\begin{document}
\title[Conformal embeddings of affine vertex algebras in  $W$-algebras]{Conformal embeddings of affine vertex algebras in minimal $W$-algebras II: decompositions}
\author[Adamovi\'c, Kac, M\"oseneder, Papi, Per\v{s}e]{Dra{\v z}en~Adamovi\'c}
\author[]{Victor~G. Kac}
\author[]{Pierluigi M\"oseneder Frajria}
\author[]{Paolo  Papi}
\author[]{Ozren  Per\v{s}e}

\begin{abstract}   We present methods for computing the explicit decomposition of the minimal simple affine $W$-algebra $W_k(\g, \theta)$  as a module for its maximal affine subalgebra $\mathcal V_k(\g^{\natural})$ 
at a conformal level $k$, that is, whenever the Virasoro vectors of $W_k(\g, \theta)$ and $\mathcal V_k(\g^\natural)$ coincide. A particular emphasis is given on the application of affine fusion rules to the determination of  branching rules. In almost all cases when  $\g^{\natural}$ is a semisimple Lie algebra, we show that, for a suitable conformal level $k$, 
$W_k(\g, \theta)$ is  isomorphic to an extension of  $\mathcal V_k(\g^{\natural})$  by  its simple module. We are able to  prove that in certain cases  $W_k(\g, \theta)$  is a simple current extension of $\mathcal V_k(\g^{\natural})$.  In order to analyze more complicated non simple current  extensions at conformal levels, we present an explicit realization of the simple $W$-algebra $W_{k}(sl(4), \theta)$ at $k=-8/3$. We prove, as conjectured in \cite{A-2014}, that $W_{k}(sl(4), \theta)$ is isomorphic to the  vertex algebra $\mathcal R^{(3)}$, and construct infinitely many singular vectors using screening operators. 
%
%
We also construct a new family of simple current modules for   the  vertex algebra $V_k (sl(n))$ at certain  admissible levels and for $V_k (sl(m \vert n)), m\ne n, m,n\geq 1$ at arbitrary levels.
\end{abstract}

\date{\today}
\maketitle
 \tableofcontents 
\section{Introduction}\label{zero}

 Let $\g=\g_{\ov 0}\oplus \g_{\ov 1}$ be a basic simple Lie superalgebra. Choose a Cartan subalgebra $\h$ for $\g_{\bar 0}$  and let $\D$ be the set of roots. Choose a subset of positive roots such that  the minimal root $-\theta$ is even. 
 Let $W_k(\g, \theta)$ be the  simple minimal affine $W$-algebra at level $k$ associated to $(\g,\theta)$ \cite{KW}, \cite{KW2}, \cite{AKMPP1}. Let $\mathcal V_k(\g^\natural)$ be its maximal affine subalgebra 
 (see \eqref{gnatural} and Definition \ref{NG}). In \cite{AKMPP1} we classified the levels $k$ such that $\mathcal V_k(\g^{\natural} )$ is conformally embedded into  $W_k(\g, \theta)$, i.e. such that the Virasoro vectors of $W_k(\g, \theta)$ and $\mathcal V_k(\g^\natural)$ coincide. We 
 call these levels the {\sl conformal levels}. We proved that, if $k$ is a conformal level,  then $W_k(\g, \theta)$ either collapses to $\mathcal V_k(\g^\natural)$ or 

$$k  \in \{ -\frac{2}{3}h ^{\vee},   -\frac{h ^{\vee}-1}{2} \}. $$
(see Section \ref{review-conformal} with review of main results from \cite{AKMPP1}).

In the present paper, we study the decomposition of $W_k(\g, \theta)$ as a $\mathcal V_k(\g^{\natural} )$-module when $k$ is a conformal level.
 It turns out that  we can use methods similar  to those we developed for studying conformal embeddings of affine vertex algebras in \cite{AKMPP} (see also \cite{A}, \cite{KMP}). Our main technical tool  is  the representation theory of affine vertex algebras at admissible and negative integer levels,  in particular  fusion rules for $\mathcal V_k(\g^{\natural} )$-modules, as we shall explain below.
\par
As in \cite{AKMPP}, it is natural to discuss the case in which $\g^\natural$ is semisimple separately from the case in which $\g^\natural$ has a nontrivial center.
In the latter case, one has the eigenspace decomposition for the action of the center:
$$ {W}_{k}(\g, \theta)  \cong \bigoplus _{i \in \Z} {W}_{k}(\g, \theta) ^{(i) },$$ 
and we  prove the following result (see  Theorems  \ref{finite-dec} and \ref{A3special}).

 \begin{theorem} \label{finite-dec-int} Consider the  conformal embeddings of  $ \mathcal V_k(\g^\natural) $ into ${W}_{k} (\g,\theta) $:
\begin{enumerate}
\item${\mathfrak g} =sl(n) $ or   ${\mathfrak g}=sl(2|n) $  ($n \ge 4$), ${\mathfrak g}=sl(m|n) $, $m>2$,  $m\ne n+3,n+2, n, n-1$, conformal level $k= -\frac{h^\vee-1}{2}$;\label{A-emb1-int}
\item
${\mathfrak g} =sl(n)$  ($n= 5$  or $n \ge 7$), ${\mathfrak g}=sl(2|n) $ $(n \ge 3)$,  ${\mathfrak g}=sl(m|n) $, $ m>2$, $m\ne n+6,n+4,n+2,n$,  conformal level $k=-\frac{2}{3}h^\vee$. \label{A-emb2-int}
\item${\mathfrak g} =sl(4) $,  conformal level  $k=-\tfrac{2}{3} h^{\vee}= -\frac{8}{3}$. \label{A-emb3-int}
\end{enumerate}
Then $\mathcal V_k(\g^\natural)$ is simple and,
  in  cases (\ref{A-emb1-int}), (\ref{A-emb2-int}), each ${W}_{k}(\g, \theta) ^{(i) } $ is an   irreducible  $\mathcal V_{k}(\g^ {\natural} )$-module, while in  case (\ref{A-emb3-int}) each  $ {W}_{k}(\g, \theta) ^{(i) } $ is an infinite sum of irreducible $\mathcal V_{k}(\g^ {\natural} )$-modules.
   \end{theorem}

It is quite surprising that there is only  one case when each  $ {W}_{k}(\g, \theta) ^{(i) } $ is an infinite sum of irreducible $\mathcal V_{k}(\g^ {\natural} )$-modules, namely the conformal embedding of 
$\mathcal V_k (gl(2))$ into $W_k (sl(4), \theta)$ and $k =-\frac{2}{3} h^{\vee} =-8/3$. This is related to  the  new explicit  realization of $W_k (sl(4), \theta)$,  conjectured in \cite{A-2014}, 
as the vertex operator algebra $\mathcal R^{(3)}$. This conjecture is proven in  Section  \ref{sect-r3} of the present paper. Our realization is based on  the logarithmic extension of the Wakimoto modules for $V_{-5/3} (gl(2))$ by singular vectors  constructed   using screening operators. The simplicity of $\mathcal R^{(3)}$ is proved by constructing certain relations in the corresponding Zhu algebra. A generalization of this construction, together with more applications related to vertex algebras appearing in LCFT  will be considered in \cite{A-2016}.\par

Theorem \ref{finite-dec-int} is proved by studying fusion rules for $\mathcal V_{k}(\g^ {\natural} )$-modules.  Recall that fusion rules (or fusion coefficients) are the dimensions of certain spaces of intertwining operators. The fact that a certain fusion coefficient is zero implies that a given $\mathcal V_{k}(\g^ {\natural} )$-module cannot appear in the decomposition  
of $W_k(\g, \theta)^{(i)}$. In many cases this is the only  information we need    to establish both the simplicity of  $\mathcal V_k(\g^{\natural} )$
and the semisimplicity of $W_k(\g, \theta)$ as a $\mathcal V_k(\g^{\natural} )$-module: see Theorem  \ref{semicenter}. This information, very often, is   obtained just  from  the decomposition of  tensor products of certain simple finite dimensional $\g^ {\natural}$-modules. \par
In the cases when we are able to compute precisely the  fusion rules, we are also able  to describe explicitly the decomposition of $W_k(\g, \theta)$ as a $\mathcal V_k(\g^{\natural} )$-module. Indeed,  we  prove that  the modules    ${W}_{k}(\g, \theta) ^{(i) } $ are simple currents   in a suitable category of $\mathcal V_{k}(\g^ {\natural} )$-modules in the following instances: 
 either  when    $\mathcal V_{k}(\g^ {\natural} )$ contains a subalgebra which is  an admissible affine  vertex algebra of type $A$ (cf. Theorem \ref{finite-dec-refined}), or  in the case of the affine $W$-algebras $W_{-2}(sl(5), \theta)$ (cf. Corollary \ref{sl5}), $W_{-2}(sl(n+5|n),\theta)$ (cf. Corollary \ref{sl6,1}, Remark \ref{85}).
We believe that this property of the modules  ${W}_{k}(\g, \theta) ^{(i) } $ holds in   all cases (\ref{A-emb1-int})  and (\ref{A-emb2-int})  from Theorem \ref{finite-dec-int}.
\par
As a byproduct, we construct a new family of simple current modules for the vertex superalgebra $V_k(sl(m \vert n))$ which belong to the category $KL_k$ (cf. Theorem \ref{sc-1}).

Next we consider the cases when $\g ^{\natural}$ is a semisimple Lie algebra. Then we have a natural decomposition
$$ {W}_{k} (\g,\theta)  =  {W}_{k} (\g,\theta) ^ {\bar 0 } \oplus {W}_{k} (\g,\theta) ^ {\bar 1 },   $$
according to the parity of twice the conformal weight.
\newline\noindent The subspaces ${W}_{k} (\g,\theta)^ {\bar i }$ are naturally  $\mathcal V_{k}(\g^ {\natural})$-modules, so we are reduced  to compute the decompositions of the ${W}_{k} (\g,\theta)^ {\bar i }$.
We solve our problem in many cases  by showing that the subspaces  ${W}_{k} (\g,\theta) ^ {\bar i }$ are actually irreducible  as  $\mathcal V_{k}(\g^ {\natural})$-modules:
\begin{theorem} \label{finite-dec-shorter-int}
Consider the following  conformal embeddings of  $ \mathcal V_k(\g^\natural) $ into ${W}_{k} (\g,\theta) $ \cite{AKMPP1}: 
\begin{enumerate}
\item ${\mathfrak g}=so(n) $  ($n \ge 8$, $n\ne 11$),   $\g=osp(4|n)$  ($n \ge 2$), $\g=osp(m|n)$ ($m\geq 5$, $m\ne n+r$, $r\in\{-1,2,3,4,6,7,8,11\}$) or $\g$ is of type $G_2$, $F_4$ $E_6$, $E_7$, $E_8$  or $\g$ is a Lie superalgebra 
$D(2,1,a)$ $(a = 1, 1/4)$, 
$ F(4)$  and $k =-\frac{h^\vee-1}{2} $.
\item ${\mathfrak g}=sp(n) $( $n \ge 6$), $\g=spo(2|m)$ ($m \ge 3$, $m\ne4$), $\g=spo(n|m)$ ($n\geq 4$, $m\ne n+2,n,n-2,n-4$) and $k=-\tfrac{2}{3} h^{\vee}$.
\end{enumerate}
Then $\mathcal V_k(\g^\natural)$ is simple and
 each
${W}_{k} (\g,\theta) ^ {\bar i }$, $i=0,1$, is   an irreducible $\mathcal V_{k} ({\mathfrak g}^{\natural})$-module.
\end{theorem}
In most cases, the proof of this theorem follows quite easily combining information on  fusion rules  coming from tensor product  decompositions of $\g^\natural$-modules with  basic principles of  vertex algebra theory, such as Galois theory: see Theorem \ref{finite-dec-shorter}.
In some cases our proof is technically involved  and  it uses the representation theory of the admissible affine vertex algebra $V_{-1/2}(sl(2))$: see Theorem \ref{exeptionalfinite}. \par We also believe that the conformal embeddings listed in Theorem \ref{finite-dec-shorter-int} provide all cases where $ {W}_{k} (\g,\theta)$ is isomorphic to an extension of $\mathcal V_{k} ({\mathfrak g}^{\natural})$ by its simple module. This result is also interesting from the perspective of extensions of affine vertex algebras, since it gives a realization of a broad list of extensions which (so far) cannot be constructed by other methods.

In our forthcoming papers we shall consider embeddings of $\mathcal V_k (\g ^{\natural})$ into $W_k(\g, \theta)$ when  critical levels appear.



 
  \vskip 5pt
  
{\bf Acknowledgments.}  
Dra{\v z}en~Adamovi\'c  and Ozren~Per\v{s}e are partially supported by the Croatian Science Foundation under the project 2634 and by the Croatian Scientific Centre of Excellence  QuantixLie. Pierluigi M\"oseneder Frajria and Paolo Papi are partially supported by PRIN project ``Spazi di Moduli e Teoria di Lie''.

 \vskip 5pt
  {\bf Notation.} The base field is $\C$. As usual, tensor product of a  family of vector spaces ranging over the empty set is meant to be $\C$. For a vector superspace $V=V_{\bar 0}\oplus V_{\bar 1}$ we set  $\Dim V= \dim V_{\bar 0} |\dim V_{\bar 1}$ and $\sdim V= \dim V_{\bar 0} - \dim V_{\bar 1}$.
  

\section{Preliminaries}

\subsection{Vertex algebras} Here we fix notation about vertex algebras (cf. \cite{KacV}). Let $V$ be a vertex algebra, with vacuum vector $\vac$.
 The vertex operator corresponding to the state $a\in V$ is denoted by
$$
Y(a,z)=\sum_{n\in\ganz} a_{(n)}z^{-n-1}, \ \text{ where $a_{(n)}\in End\,V$}.
$$
 We frequently use the notation of $\l$-bracket and normally ordered product:
$$[a_\l b]=\sum_{ i\geq 0}�\frac{\l^i}{i!} (a_{( i )} b),\qquad : ab: = a_{(-1)} b.$$ 
If $V$ admits a Virasoro(=conformal) vector and $\D_a$ is the conformal weight of $a$, then we also write the corresponding vertex operator as 
$$
Y(a,z)=\sum_{m\in\ganz-\D_a} a_{m}z^{-m-\D_a},
$$ 
so that 
$$
a_{(n)} =a_{n-\D_a+1},\  n\in\ganz,\quad  a_m =a_{(m+\D_a-1)},\ m\in\ganz-\D_a.
$$
The product  of two subsets $A,B$ of $V$ is 
$$
A\cdot B=span(a_{(n)}b\mid n\in\ganz,\ a\in A,\ b\in B).
$$
This product is associative (cf. \cite{BK2}).
\subsection{Intertwining operators and fusion rules}\label{fr}
Let $V$ be a vertex algebra. A $V$-module   is  a vector  superspace $M$  endowed with a  parity preserving map $Y^M$ from
$V$ to the superspace of $End(M)$-valued fields
$$ a\mapsto Y^M(a,z)=\sum_{n\in\ganz}a^M_{(n)}z^{-n-1}
$$ such that
\begin{enumerate}
\item $Y^M(|0\rangle,z)=I_M$,
\item for $a,b \in V$, $m,n,k \in \ganz$,
\begin{align*}& \sum_{j\in\nat}\binom{m}{j}(a_{(n+j)}b)^M_{(m+k-j)}\\
&=\sum_{j\in\nat}(-1)^j\binom{n}{j}(a^M_{(m+n-j)}b^M_{(k+j)}-p(a,b)(-1)^nb^M_{(k+n-j)}a^M_{(m+j)}),
\end{align*}
\end{enumerate}

Given three $V$-modules $M_1$, $M_2$, $M_3$, an {\it intertwining operator of type $\left[\begin{matrix}M_3\\M_1\ M_2\end{matrix}\right]$} (cf. \cite{FZ}, \cite{Lep}) is a   map $I:a\mapsto I(a,z)=\sum_{n\in \ganz}a^I_{(n)}z^{-n-1}$ from
$M_1$ to the space of $End(M_2,M_3)$-valued fields such that
for $a \in V$, $b \in M_1$, $m,n \in \ganz$,
\begin{align*}& \sum_{j\in\nat}\binom{m}{j}(a^{M_1}_{(n+j)}b)^I_{(m+k-j)}\\
&=\sum_{j\in\nat}(-1)^j\binom{n}{j}(a^{M_3}_{(m+n-j)}b^I_{(k+j)}-p(a,b)(-1)^nb^I_{(k+n-j)}a^{M_2}_{(m+j)}).
\end{align*}
We let $I { M_3 \choose M_1 \ M_2}$ denote the space of intertwining operators of type $\left[\begin{matrix}M_3\\M_1\ M_2\end{matrix}\right]$, and set 
$$N^{M_3}_{M_1,M_2}=\dim I\left(\begin{matrix}M_3\\M_1\ M_2\end{matrix}\right).$$ When $N^{M_3}_{M_1,M_2}$ is finite, it is  usually called a {\it fusion coefficient}.\par

 Assume that in a category 
$K$ of $\Z_{\ge 0}$--graded $V$-modules, the irreducible modules  $\{M _i \ \vert \ i \in I\}$, where $I$ is an index set, have the following properties
\begin{itemize}
\item[(1)]
  for every $i, j \in I$  $N^{M_k}_{M_i,M_j}$ is finite for any $k\in I$;
  \item[(2)]  $N^{M_k}_{M_i,M_j} = 0 $ for all but finitely many $k \in I$.
  \end{itemize}
  
Then the algebra with basis $\{e_i  \in I \}$ and product
$$e_i\cdot e_j = \sum_{k\in I } N^{M_k}_{M_i,M_j}  e_k$$
is called the {\it fusion algebra} of $V, K$. (Note that we consider only fusion rules between irreducible modules).

\vskip5pt

Let $K$ be a category of  $V$-modules. Let $M_1$, $M_2$ be irreducible $V$-modules in $K$. Given an irreducible $V$-module $M_3$ in $K$, we will say that the fusion  rule
\begin{equation}\label{fr-sc}
  M_1 \times M_2=M_3
  \end{equation}
holds in $K$ if $N^{M_3}_{M_1,M_2}=1$ and  $N^{R}_{M_1,M_2}=0$
   for any other irreducible $V$-module  $R$ in $K$ which is not isomorphic to $M_3$.

We say that an irreducible $V$-module  $M_1$ is a simple current in $K$   if $M_1$ is in $K$ and, for every irreducible $V$-module  $M_2$ in $K$, there is an irreducible $V$-module $M_3$ in $K$ such that the fusion rule \eqref{fr-sc} holds in $K$
 (see \cite{DLM}).
 

\subsection{Commutants} Let $U$ be a vertex subalgebra of $V$.  Then the commutant of $U$ in $V$ (cf. \cite{FZ}) is the following vertex subalgebra of $V$: 
$$ \mbox{Com} (U, V) =\{ v \in V  \ \vert \   [Y(a,z) , Y(v,w) ] = 0, \ \forall  a\in U\}. $$
In the case when $U$ is an affine vertex algebra, say $ U= V_k (\g)$ (see below),   it is easy to see that
$$ \mbox{Com} (U, V) =\{ v \in V  \ \vert \    x_{(n)}v = 0 \ \forall x \in \g, \ n\ge 0 \}. $$

 \subsection{Zhu algebras}  \label{sect-zhu}

Assume that $V$ is a vertex algebra, endowed with  a  conformal vector $\omega$ such that the conformal weights $\D_v$  of $v\in V$ are in $\frac{1}{2} {\Z}$. Then 
$$ V = \bigoplus _{r  \in \frac{1}{2} {\Z}} V (r), \qquad V(r)=\{v\in V\mid \D_v=r\}.$$
Set $$ V^{\bar 0} = \bigoplus _{r \in \Z} V(r), \quad V^{\bar 1} = \bigoplus _{r \in \frac{1}{2} + {\Z} } V(r). $$

We define two bilinear maps $* : V  \times V \rightarrow V$,
$\circ : V \times V \rightarrow V$ as follows: for homogeneous $a,
b \in V$, let
\begin{align}
a* b &= \begin{cases}\mbox{Res}_x Y(a,x) \frac{(1+x)^{\D_a}}{x}b  & \text{if  $a,b  \in V^{\bar{0}}$,}\\
  0  &\text{if $a$  or $b\in V^{\bar{1}}$}\end{cases}
\\
a\circ b &= \begin{cases}
 \mbox{Res}_x Y(a,x) \frac{(1+x) ^{\D_a} }{x^2}b  & \text{if $a  \in V^{\bar{0}}$,}\\
  \mbox{Res}_x Y(a,x) \frac{(1+x) ^{\D_a -\frac{1}{2} }}{x}b  & \text{if  $a   \in V^ {\bar{1}}.$} \end{cases}
\end{align}
Next, we extend $*$ and $\circ$ to $V \otimes V$ linearly, and
denote by $O(V)\subset V$ the linear span of elements of the form
$a \circ b$, and by $A(V)$ the quotient space $V / O(V)$. The
space $A(V)$ has a unital associative algebra structure, with  multiplication induced by $*$. The algebra $A(V)$ is called  the
Zhu algebra of $V$. The image of $v \in V$, under the natural
map $V \mapsto A(V)$ will be denoted by $[v]$.
In the case when $V ^{\bar 0} = V$  we get the usual definition of Zhu algebra for vertex operator algebras.

An important result, proven by Zhu \cite{Zhu}  for $\ganz$-graded $V$  and later by Kac-Wang \cite{KWn} for $\tfrac{1}{2}\ganz$-graded $V$,   is the following theorem.

\begin{theorem} \label{zhu-corr-mod} There is a one-to-one correspondence between  irreducible $\tfrac{1}{2}\ganz_{\ge 0}$-graded $V$-modules and
irreducible $A(V)$-modules.
\end{theorem}

\subsection{Affine vertex algebras} Let $\aa$ be a Lie superalgebra equipped with a nondegenerate  invariant supersymmetric bilinear form $B$.  The universal affine vertex algebra $V^B(\aa)$ is  the universal enveloping vertex algebra of  the  Lie conformal superalgebra $R=(\C[T]\otimes\aa)\oplus\C$ with $\lambda$-bracket given by
$$
[a_\lambda b]=[a,b]+\lambda B(a,b),\ a,b\in\aa.
$$
In the following, we shall say that a vertex algebra $V$ is an {\sl affine vertex algebra} if it is a quotient of some $V^B(\aa)$.

 If $\aa$ is  simple or is a one dimensional abelian Lie algebra, then one usually fixes a nondegenerate  invariant supersymmetric bilinear form $(\cdot | \cdot)$ on $\aa$. Any invariant supersymmetric bilinear form is therefore a constant multiple of $(\cdot|\cdot)$. In particular, if $B=k (\cdot | \cdot)$ ($k\in\C$), then we denote  $V^B(\aa)$ by  $V^k(\aa)$. Let $h_\aa^\vee$ be half of the eigenvalue of the Casimir element corresponding to $(\cdot | \cdot)$; if $h_\aa^\vee\ne -k$, then    $V^k(\aa)$
admits a  unique irreducible quotient  which we denote by $V_k(\aa)$.
\par In the same hypothesis,  ${V}^{k}(\aa)$ is equipped with a Virasoro vector 
\begin{equation}\label{sug}
\omega^\aa_{sug}=\frac{1}{2(k+h_\aa^\vee)}\sum_{i=1}:b_i a_i:,
\end{equation}
where $\{a_i\}$ is a basis of $\aa$ and $\{b_i\}$ is its dual basis with respect to $(\cdot|\cdot)$.
If a vertex algebra $V$ is some quotient of $V^k(\aa)$, we will say that $k$ is the level of $V$.
\par 
A module $M$ for $\widehat\aa$ is said to be of level $k$ if $K$ acts on $M$ by $kI_M$.
Finally recall that an irreducible highest weight module $L(\l),$
over an affine algebra  $\widehat\aa$ is said to be {\sl admissible}   \cite{KacWW}
if
\begin{enumerate}
\item $(\l+\rho)(\a)\notin\{0,-1,-2,\ldots\}$, for each positive coroot $\a$;
\item the rational linear span of positive simple coroots equals the rational linear span of the coroots which are integral valued on $\l+\rho$.
\end{enumerate}
\subsection{Heisenberg vertex algebras}\label{hva}
In the special case when $\aa$ is an abelian Lie algebra,  $V^B(\aa)$ is of course a Heisenberg vertex algebra. If this is the case, we will denote  $V^B(\aa)$ by $M_\aa(B)$. In the special case when  $\aa$ is one dimensional, then we can choose a basis $\{\a\}$ of $\aa$ and the form $(\cdot|\cdot)$ so that $(\a|\a)=1$. With these choices we denote  $V^k(\aa)$ by $M_\a(k)$ or simply by $M(k)$ when the reference to the generator $\a$ need not to be explicit. The vertex algebra $M(k)$ is called the universal Heisenberg vertex algebra of level $k$ generated by $\a$. Recall that, if $k\ne0$,  $M_\a(k)$ is simple and  that $M (k) \cong M (1)$. The irreducible $M_\a(k)$-modules are the modules $M_\a(k,s)$ (or simply $M(k,s)$) generated by a vector $v_s$ with action, for $n\in\ganz_+$, given by  $\a_{(n)}v_s=\delta_{n,0} sv_s$. 

The irreducible modules of the  Heisenberg vertex algebra  $M(k)$ are all simple currents in the category of $M(k)$-modules. Indeed we have the following fusion rules (cf.\cite{FZ}):
\bea && M(k, s_1) \times  M(k, s_2) = M(k, s_1 + s_2) \qquad (s_1, s_2 \in {\C}) .  \label{fusion-heisenberg} \eea

If $\aa$ is simple or one-dimensional even abelian with fixed nondegenerate  invariant supersymmetric bilinear form $(\cdot | \cdot)$, the affinization of $\aa$ is the Lie superalgebra $\widehat \aa= (\C[t,t^{-1}]\otimes\aa)\oplus \C d\oplus \C K$ where
$K$ is a central element, and $d$ acts as $td/dt$. We choose the central element $K$ so that 
$$
[t^s\otimes x,t^r\otimes y]=t^{s+r}\otimes [x,y]+\delta_{r,-s}rK(x|y).
$$
 Let $\h$ be a Cartan subalgebra of $\aa$ and $\ha=\h\oplus \C K\oplus \C d$ a Cartan subalgebra of $\widehat \aa$. Let $\Lambda_0\in \ha^*$ be defined by $\Lambda_0(K)=1$, $\Lambda_0(\h)=\Lambda_0(d)=0$.
 We fix a set of simple roots for $\widehat\aa$ and denote by $\rho\in\ha^*$ a corresponding Weyl vector.
We shall denote by $L_\aa(\l)$ the irreducible highest weight   $V^k(\aa)$-module of  highest weight $\l\in \ha^*$. Sometimes, if no confusion may arise, we simply write $L(\l)$.
Similarly, we shall denote by $V_\aa(\l)$ or simply by $V(\l)$ the irreducible highest weight $\aa$-module of highest weight $\l\in\h^*$. Note that in the case when $\aa$  is a one dimensional abelian Lie algebra   $\C\a$, then $M_\a(k,s)=L(k\L_0+s\l_1)$ where $\l_1\in(\C\a)^*$ is defined by setting $\l_1(\a)=1$.


\subsection{Rank one lattice vertex algebra $ V_{\Z\alpha} $}\label{Rankonelattice}



Assume that $L$ is an integral, positive definite lattice; let $L^\circ$ be its dual lattice. Set $V_L $ to be the lattice vertex algebra  (cf. \cite{KacV})
associated to $L$. 

The set of isomorphism classes of irreducible $V_L$-modules is parametrized by $L^\circ/L$   (cf. \cite{Dong}). Let $V_{\bar \l}$ denote the irreducible $V_L$-module corresponding to $\bar\l=\l+L\in L^\circ/L$.
Every irreducible $V_L$-module is a simple current.

We shall now consider rank one lattice vertex algebras.
For $n \in {\ZZ} _{> 0}$, let $M(n)$ be the universal Heisenberg vertex algebra generated by $\a$ and let $F_n$ denote the lattice vertex algebra $V_{\ZZ\alpha} =M(n) \otimes {\mathbb C}[\Z\alpha] $   associated to the lattice $ L=\Z\alpha $, $\langle \alpha, \alpha \rangle = n$. 
The dual lattice of $L$ is
$ L^o = \frac{1}{n} L$. 
 For $i\in \{0, \dots, n-1\}$, set
$F_n ^i  =V_{\frac{i}{n} \alpha + \Z \alpha} $.  Then the set 
$$ \{ F_n ^i  \ \vert \ i=0, \dots, n-1\}$$  provides a complete list of non-isomorphic irreducible $F_n$-module. We choose the following Virasoro vector in $F_n$: 
$$\omega_{F_n} = \frac{1}{2n}: \alpha \alpha:. $$

As a $M(n)$-module,  $F_n$ decomposes as
\bea 
F_n = \bigoplus_{ j  \in \Z} M  (n ) e^{j \alpha} = \bigoplus_{ j  \in \Z} M  (n, j n)  \label{decomposition}
\eea
The following result is a consequence of the  result of H. Li and X. Xu  \cite{LX}  on characterization of lattice vertex algebras.
\begin{proposition}  \label{lattice-characterization} Assume that $V =\bigoplus_{i   \in \Z} V_i$ is a  $\Z$-graded vertex algebra  satisfying the following properties
\begin{itemize}
\item[(1)] $V$ is a subalgebra of a simple vertex algebra $W$;
\item[(2)]   there exists a Heisenberg vector
$\a\in V_0$ such that  $V_0= M_{\a}(n)$,  and $V_i \cong M_{\a} (n , i n)$ as a $V_0$-module.
\end{itemize}
Then  $V$ is a simple vertex algebra and $V \cong F_n$.
\end{proposition}
\begin{proof} The  Main Theorem of  \cite{LX} implies  that a simple vertex algebra satisfying condition (2) is isomorphic to $F_n$.  
To prove simplicity, we first observe that 
\bea Y(v,z) w \ne 0\quad\text{ $\forall\,v, w \in V$},\label{non-triv-1}\eea 
which in our setting holds  since $W$ is simple. Now (\ref{non-triv-1}) and the fusion rules   (\ref{fusion-heisenberg}) imply that
$$ V_i \cdot V_j = V_{i+j} \qquad (i,j \in \Z). $$
This implies that $V$ is  simple, and the claim follows.
\end{proof}

 \section{Minimal quantum affine $W$-algebras }\label{uno}
 
 In this section we briefly recall some results of  \cite{KW} and    \cite{AKMPP1}. 
 We include an example which contains explicit $\lambda$--bracket formulas for $W^k(sl(4), \theta)$ which we shall need in Section \ref{sect-r3}. 
 
We first  recall the definition of minimal affine $W$-algebras.
 
Let $\g$ be a basic simple Lie superalgebra. Choose a Cartan subalgebra $\h$ for $\g_{\bar 0}$  and let $\D$ be the set of roots. Fix a minimal root $-\theta$ of $\g$. (A root $-\theta$ is called {\sl minimal} if it is even and there exists an additive  function $\varphi:\D\to \R$ such that $\varphi_{|\D}\ne 0$ and $\varphi(\theta)>\varphi(\eta),\,\forall\,\eta\in\D\setminus\{\theta\}$).  We choose root vectors $e_\theta$ and $e_{-\theta}$ such that 
$$[e_\theta, e_{-\theta}]=x\in \h,\qquad [x,e_{\pm \theta}]=\pm e_{\pm \theta}.$$
Due to the minimality of $-\theta$, the eigenspace decomposition of $ad\,x$ defines a {\sl minimal} $\frac{1}{2}\ganz$-gradation (\cite[(5.1)]{KW}):
\begin{equation}\label{gradazione}
\g=\g_{-1}\oplus\g_{-1/2}\oplus\g_{0}\oplus\g_{1/2}\oplus\g_{1},
\end{equation}
where $\g_{\pm 1}=\C  e_{\pm \theta}.$  One has
\begin{equation}\label{gnatural}
\g_0=\g^\natural\oplus \C x,\quad\g^\natural=\{a\in\g_0\mid (a|x)=0\}.
\end{equation}
For a given choice of a minimal root $-\theta$, we normalize the invariant bilinear form $( \cdot | \cdot)$ on $\g$ by the condition
\begin{equation}\label{normalized}
(\theta | \theta)=2.
\end{equation}
The dual Coxeter number $h^\vee$ of the pair $(\g, \theta)$ is defined to be   half the eigenvalue of the Casimir operator of $\g$ corresponding to $(\cdot|\cdot)$.
\par
 The complete list of the Lie superalgebras $\g^\natural$, the $\g^\natural$-modules $\g_{\pm 1/2}$ (they are isomorphic and  self-dual),  and $h^\vee$ for all possible choices of $\g$ and of $\theta$ (up to isomorphism)  is given in Tables  1,2,3 of \cite{KW}, 
 and it is as follows
 
 \vskip10pt
{\tiny
\centerline{Table 1}
\vskip 5pt
\noindent {\sl $\g$ is a simple Lie algebra.}
\vskip 5pt
\centerline{\begin{tabular}{c|c|c|c||c|c|c|c}
$\g$&$\g^\natural$&$\g_{1/2}$&$h^\vee$&$\g$&$\g^\natural$&$\g_{1/2}$&$h^\vee$\\
\hline
$sl(n), n\geq 3$&$gl(n-2)$&$\C^{n-2}\oplus (\C^{n-2})^* $&$n$&$F_4$&$sp(6)$&$\bigwedge_0^3\C^6$ & $9$\\\hline 
$so(n), n\geq 5$&$sl(2)\oplus so(n-4)$&$\C^2\otimes\C^{n-4}$&$n-2$&$E_6$&$sl(6)$&$\bigwedge^3\C^6$ & $12$\\\hline
$sp(n), n\geq 2$&$sp(n-2)$&$\C^{n-2} $&$n/2+1$&$E_7$&$so(12)$&$spin_{12}$ & $18$\\\hline
$G_2$&$sl(2)$&$S^3\C^2$&$4$&$E_8$&$E_7$&$\dim=56$ & $30$\\
\end{tabular}}
\vskip 25pt
\centerline{Table 2}
\vskip 5pt
\noindent {\sl $\g$ is not a Lie algebra but $\g^\natural$ is and $\g_{\pm1/2}$ is purely odd ($m\ge1$).}
\vskip 5pt
\centerline{\begin{tabular}{l|c|c|c||c|c|c|c}
$\g$&$\g^\natural$&$\g_{1/2}$&$h^\vee$&$\g$&$\g^\natural$&$\g_{1/2}$&$h^\vee$\\
\hline
$sl(2|m),$&$gl(m)$&$\C^{m}\oplus (\C^{m})^* $&$2-m$&$D(2,1;a)$&{\tiny $sl(2) \oplus sl(2)$}&$\C^2\otimes \C^2$ & $0$\\
$m\ne2$& & & & & & \\\hline
$psl(2|2) $&$sl(2)$&$\C^2\oplus\C^{2}$&$0$&$F(4)$&$so(7)$&$spin_7$ & $-2$\\\hline
$spo(2|m)$&$so(m)$&$\C^{m} $&$2-m/2$&$G(3)$&$G_2$&$\Dim= 0|7$ & $-3/2$\\\hline
$osp(4|m)$&$sl(2)\oplus sp(m)$&$\C^2\otimes \C^m$&$2-m$\\
\end{tabular}}
\vskip 25pt
\centerline{Table 3}
\vskip 5pt
\noindent {\sl Both $\g$ and $\g^\natural$ are  not  Lie algebras ($m,n\geq 1$).}
\vskip 5pt
\centerline{\begin{tabular}{c|c|c|c}
$\g$&$\g^\natural$&$\g_{1/2}$&$h^\vee$\\
\hline
$sl(m|n)$, $m\neq n, m>2$&$gl(m-2|n)$&$\C^{m-2|n}\oplus(\C^{m-2|n})^*$&$m-n$\\
\hline
$psl(m|m),\,m>2$&$sl(m-2|m)$& $\C^{m-2|m}\oplus(\C^{m-2|m})^*$&$0$\\
\hline
$spo(n|m),\,n\ge 4$& $spo(n-2|m)$ &$\C^{n-2|m}$&$1/2(n-m)+1$\\
\hline
$osp(m|n),\,m\geq 5$&$osp(m-4|n)\oplus sl(2)$ &$\C^{m-4|n}\otimes \C^2$&$m-n-2$\\
\hline
 $F(4)$&$D(2,1;2)$ &$\Dim=6|4$& $3$\\
 \hline
$G(3)$&$osp(3|2)$ &$\Dim=4|4$& $2$\\
\end{tabular}}
}
\vskip10pt
 In this paper we shall exclude the case of $\g=sl(n+2|n)$, $n>0$. In all other cases the Lie superalgebra $\g^\natural$ decomposes in a direct sum of ideals, called components of $\g^\natural$:
\begin{equation}\label{decompgnat}
\g^\natural=\bigoplus_{i\in I}\g^\natural_i,
\end{equation}
where each  summand is either the (at most 1-dimensional) center of $\g^\natural$ or is a basic classical simple Lie superalgebra different from $psl(n|n)$. We will also exclude $\g=sl(2)$.

 It follows from the tables that the index set $I$ has cardinality $r=0$, $1$, $2$, or $3$. The case $r=0$, i.e.\ $\g^\natural=\{0\}$, happens if and only if $\g=spo(2|1)$. In the case when the center is non-zero (resp. $\{0\}$) we use $I=\{0,1,\ldots,r-1\}$ (resp. $I=\{1,\ldots,r\}$) as the index set, and denote the center of $\g^\natural$ by $\g^\natural_0$.

Let $C_{\g^\natural_i}$  be the Casimir operator of $\g^\natural_i$ corresponding to $(\cdot|\cdot)_{|\g^\natural_i\times \g^\natural_i}$. We define the dual Coxeter number $h^\vee_{0,i}$ of $\g_i^\natural$ as half of the eigenvalue of $C_{\g^\natural_i}$  acting on $\g^\natural_i$ (which is $0$ if $\g_i^\natural$ is abelian). Their values are given in Table 4 of \cite{KW}.


\vskip10pt
Let $W^{k} (\g, e_{-\theta})$ be the minimal W-algebras of level $k$ studied in  \cite{KW}.  It is known that, for $k$ non-critical, i.e., $k\ne - h^\vee$, the vertex algebra 
$W^{k} (\g, e_{-\theta})$  has a unique simple quotient, denoted by $W_{k} (\g, e_{-\theta})$.\par To simplify notation, we set 
\begin{align*}
{W}^{k}(\g, \theta)=W^{k}(\g, e_{-\theta}),\quad
{W}_{k}(\g, \theta)=W_{k}(\g, e_{-\theta}).
\end{align*}

Throughout  the paper we shall assume that  $k\ne - h^\vee$. In such a case, it is  known that ${W}^{k}(\g,f)$ has  a Virasoro vector $\omega$, \cite[(2.2)]{KW} that  has central charge \cite[(5.7)]{KW}
\begin{equation}\label{cgk}
c(\g,k)=\frac{k\,\sdim\g}{k+h^\vee}-6k+h^\vee-4.
\end{equation}

It is proven in \cite{KW} that the universal minimal W-algebra $W^k(\g,\theta)$
is freely and strongly generated by the elements  $J^{\{a\}}$ ($a$ runs over a basis of $\g^\natural$), $G^{\{u\}}$ ($u$ runs over a basis of $\g_{-1/2}$), and the Virasoro vector $\omega$.
Furthermore the elements $J^{\{a\}}$ (resp. $G^{\{u\}}$) are primary of conformal weight $1$ (resp. $3/2$), with respect  to $\omega$. The $\lambda$-brackets satisfied by these generators have been given in \cite{KW} and, in a simplified form, in \cite{AKMPP1}.  This  simplified form reads:
\begin{align}\label{JJ}
[{J^{\{a\}}}_{\lambda}J^{\{b\}}]&=J^{\{[a,b]\}}+\lambda\left((k+h^\vee/2)(a|b)-\tfrac{1}{4}\kappa_0(a,b)\right),\ a,b\in \g^\natural,\\\label{JG}
[{J^{\{a\}}}_{\lambda}G^{\{u\}}]&=G^{\{[a,u]\}},\ a\in \g^\natural,\,u\in\g_{-1/2}. 
\end{align}
\begin{align}\label{GGsimplified}
&[{G^{\{u\}}}_{\lambda}G^{\{v\}}]=-2(k+h^\vee)(e_\theta|[u,v])\omega+(e_\theta|[u,v])\sum_{\alpha=1}^{\dim \g^\natural} 
:J^{\{u^\alpha\}}J^{\{u_\alpha\}}:\\\notag
&+\sum_{\gamma=1}^{\dim\g_{1/2}}:J^{\{[u,u^{\gamma}]^\natural\}}J^{\{[u_\gamma,v]^\natural\}}:
+2(k+1)\partial J^{\{[[e_\theta,u],v]^\natural\}}\\\notag
&+ 4 \l  \sum_{i\in I} \frac{p(k)}{k_i} J^{\{[[e_\theta, u],v]_i^\natural}+2\l^2(e_\theta|[u,v])p(k)\vac.
\end{align}

Here $\kappa_0$ is the Killing form of $\g_0$;
 $\{u_\alpha\}$ (resp. $\{v_\gamma\}$) is a basis of  $\g^\natural$ (resp. $\g_{1/2}$) and $\{u^\alpha\}$ (resp. $\{u^\gamma\}$) is the corresponding dual basis w.r.t. $(\cdot|\cdot)$ (resp w.r.t. $\langle\cdot,\cdot\rangle_{\rm ne}=(e_{-\theta}|[\cdot,\cdot])$),  $a^\natural$ is the image of $a\in\g_0$ under  the  orthogonal projection of $\g_{0}$ on $\g^\natural$ , $a_i^\natural$ is the projection of $a^\natural$ on the $i$th minimal ideal $\g_i^\natural$ of $\g^\natural$, 
 $k_i=k+\frac{1}{2}(h^\vee-h^\vee_{0,i})$, and $p(k)$ is the monic polynomial given in the following table \cite{AKMPP1}:

{\tiny
 \centerline{Table 4} 
\vskip10pt
\centerline{\begin{tabular}{c|c||c|c}
$\g$&$p(k)$&$\g$&$p(k)$\\
\hline
$sl(m|n)$, $n\ne m$&$(k+1) (k+(m-n)/2)$&$E_6$&$(k+3) (k+4)$\\
\hline
$psl(m|m)$&$ k (k+1)$&$E_7$&$(k+4)(k+6)$\\
\hline
$osp(m|n)$&$(k+2) (k+(m-n-4)/2)$&$E_8$&$(k+6) (k+10)$\\
\hline
$spo(n|m)$&$(k+1/2) (k+(n-m+4)/4)$&$F_4$&$(k+5/2) (k+3)$\\
\hline
$D(2,1;a)$&$(k-a)(k+1+a)$&$G_2$&$ (k+4/3) (k+5/3)$\\
\hline
$F(4)$, $\g^\natural=so(7)$ & $(k+2/3)(k-2/3)$ &$G(3)$, $\g^\natural=G_2$ & $(k-1/2)(k+3/4)$  \\
\hline
$F(4)$, $\g^\natural=D(2,1;2)$ & $(k+3/2)(k+1)$ &$G(3)$, $\g^\natural=osp(3|2)$ & $(k+2/3)(k+4/3)$  \\
\end{tabular}}}
\vskip 15pt
Note that the linear polynomials $k_i$ always divide $p(k)$ so the coefficients in \eqref{GGsimplified} depend polynomially on $k$.

\begin{example} \label{ex-sl4}
Consider $\g=sl(4)$. Set
$$
c=\half\begin{pmatrix}1&0&0&0\\0&-1&0&0\\0&0&-1&0\\0&0&0&1\end{pmatrix}.
$$
 In this case $\g^\natural=\g^\natural_0 \oplus \g^\natural_1$ with
 $$
 \g^\natural_0=\C c,\quad \g^\natural_1=\left\{\begin{pmatrix}0&0&0\\0&A&0\\0&0&0\end{pmatrix}\mid A\in sl(2)\right\}\simeq sl(2),$$
 so $\g^\natural\simeq gl(2)$, while $\g_{-1/2}=span(e_{2,1},e_{3,1},e_{4,2},e_{4,3})$.
 
 The $\lambda$-brackets $[{G^{\{u\}}}\!_\lambda G^{\{v\}}]$ are as follows:
 \begin{align*}
& [{G^{\{e_{2,1}\}}}\!_\lambda G^{\{e_{2,1}\}}]= [{G^{\{e_{3,1}\}}}\!_\lambda G^{\{e_{3,1}\}}]=0\\& [{G^{\{e_{4,2}\}}}\!_\lambda G^{\{e_{4,2}\}}]= [{G^{\{e_{4,3}\}}}\!_\lambda G^{\{e_{4,3}\}}]=0&\\
& [{G^{\{e_{2,1}\}}}\!_\lambda G^{\{e_{3,1}\}}]=[{G^{\{e_{4,3}\}}}\!_\lambda G^{\{e_{4,2}\}}]=0&\\
 & [{G^{\{e_{2,1}\}}}\!_\lambda G^{\{e_{4,3}\}}]=2:J^{\{c\}}J^{\{e_{2,3}\}}:-(k+2)\partial J^{\{e_{2,3}\}}-\lambda2(k+2)J^{\{e_{2,3}\}}&\\
  &[{G^{\{e_{3,1}\}}}\!_\lambda G^{\{e_{4,2}\}}]=2:J^{\{c\}}J^{\{e_{3,2}\}}:-(k+2)\partial J^{\{e_{3,2}\}}-\lambda 2(k+2)J^{\{e_{3,2}\}}&\\
  & [{G^{\{e_{2,1}\}}}\!_\lambda G^{\{e_{4,2}\}}]=\\&(k+4)\omega-2:\!J^{\{e_{2,3}\}}J^{\{e_{3,2}\}}\!:-\half:\!J^{\{e_{2,2}-e_{3,3}\}}J^{\{e_{2,2}-e_{3,3}\}}\!:&\\&-\frac{3}{2}:\!J^{\{c\}}J^{\{c\}}\!:+:\!J^{\{c\}}J^{\{e_{2,2}-e_{3,3}\}}\!:+(k+1)\partial J^{\{c\}}-\frac{k}{2}\partial J^{\{e_{2,2}-e_{3,3}\}}\\&+\lambda 2(k+1)J^{\{c\}}-\lambda (k+2)J^{\{e_{2,2}-e_{3,3}\}}-\lambda^2(k+1)(k+2)\vac&\\
   & [{G^{\{e_{4,3}\}}}\!_\lambda G^{\{e_{3,1}\}}]=\\&-(k+4)\omega+2:\!J^{\{e_{2,3}\}}J^{\{e_{3,2}\}}\!:+\half:\!J^{\{e_{2,2}-e_{3,3}\}}J^{\{e_{2,2}-e_{3,3}\}}\!:&\\&+\frac{3}{2}:\!J^{\{c\}}J^{\{c\}}\!:+:\!J^{\{c\}}J^{\{e_{2,2}-e_{3,3}\}}\!:+(k+1)\partial J^{\{c\}}+\frac{k}{2}\partial J^{\{e_{2,2}-e_{3,3}\}}\\&+\lambda 2(k+1)J^{\{c\}}+\lambda (k+2)J^{\{e_{2,2}-e_{3,3}\}}+\lambda^2(k+1)(k+2)\vac.&
 \end{align*}
\end{example}

\section{A classification of conformal levels from \cite{AKMPP1}}
\label{review-conformal}

In this section  we recall the definition of conformal embeddings of affine vertex subalgebras into minimal affine $W$--algebras and review results from \cite{AKMPP1} on the classification of conformal levels.

Let $\mathcal{V}^k(\g^\natural)$ be the  subalgebra of the vertex algebra ${W}^{k}(\g, \theta)$,  generated by $\{J^{\{a\}}\mid a\in\g^\natural\}$. By \eqref{JJ}, 
$\mathcal{V}^k(\g^\natural)$ is isomorphic to a universal affine vertex algebra. More precisely, 
 the map $a\mapsto J^{\{a\}}$ extends to an isomorphism
\begin{equation}\label{current}
\mathcal{V}^k(\g^\natural)\simeq \bigotimes_{i\in I}V^{k_i}(\g_i^\natural).
\end{equation}
\begin{defi}\label{NG}
We  set $\mathcal V_k(\g^\natural)$ to be the image of $\mathcal{V}^k(\g^\natural)$ in ${W}_{k}(\g, \theta)$. 
\end{defi}
Clearly we can write
$$
\mathcal V_k(\g^\natural)\simeq \bigotimes_{i\in I} \mathcal V_{k_i}(\g_i^\natural),
$$
where $\mathcal V_{k_i}(\g_i^\natural)$ is some quotient (not necessarily simple) of $V^{k_i}(\g^\natural_i)$.
\vskip5pt
If $k_i+h^\vee_{0,i}\ne 0$,  then ${V}^{k_i}(\g^\natural_i)$ is equipped with the Virasoro vector  $\omega_{sug}^{\g^\natural_i}$ (cf. \eqref{sug}). 
If $k_i+h^\vee_{0,i}\ne 0$ for all $i$,  we set
$$\omega_{sug}=\sum_{i\in I}\omega_{sug}^{\g^\natural_i}.$$
Define
$$\mathcal K=\{k\in\C\mid k+ h^\vee\ne 0, k_i+h^\vee_{0,i}\ne 0 \text{ whenever $k_i\ne 0$}\}.$$
If $k\in \mathcal K$ we also set
$$\omega'_{sug}=\sum\limits_{i\in I:k_i\ne 0}\omega_{sug}^{\g^\natural_i}.$$ 
We define
$$
c_{sug}=
\text{ central charge of $\omega'_{sug}$.}$$

\begin{defi} Assume $k\in \mathcal K$. We say that $\mathcal V_k(\g^\natural)$ is conformally embedded in ${W}_{k}(\g, \theta)$ if $\omega'_{sug}=\omega$.
The level $k$ is called a {\sl conformal level}.\par
If ${W}_{k}(\g, \theta)=\mathcal{V}_{k}(\g^\natural)$, we say that $k$ is  a {\sl collapsing level}.
\end{defi}
\begin{rem} The above definition of conformal level is slightly more general than the one given in the Introduction. Indeed it makes sense also when $k_i=h^\vee_{0,i}=0$.\par
\end{rem}
Next we recall the classification of collapsing levels from \cite{AKMPP1}.
\begin{proposition}  \label{class-coll} \cite[Theorem 3.3]{AKMPP1}

The level $k$ is collapsing if and only if $p(k) =0$ where $p$ is the  polynomial listed in the Table 4.
\end{proposition}

\vskip 5mm

The classification of non-collapsing conformal levels is given in Section 4 of \cite{AKMPP1}. It may be summarized as follows.
\begin{proposition}\label{confnoncollaps} \ \vskip3pt \par\noindent
\textbf{(I).} Assume that $\g^\natural$ is either zero  or simple or 1-dimensional. 

If $\g=sl(3)$, or $\g=spo(n|n+2)$ with $n\ge 2$, $\g=spo(n|n-1)$ with $n\ge 2$, $\g=spo(n|n-4)$ with $n\ge 4$, then there are no non--collapsing conformal levels.
In all other cases the non-collapsing  conformal levels  are 
\begin{enumerate}
\item $k = -\frac{ h^{\vee} -1 }{2}$ if $\g$ is of type $G_2, F_4, E_6, E_7, E_8, F(4) (\g^\natural=so(7)),G(3)$ ($\g^\natural= G_2, osp(3|2)$), or $\g=psl(m|m)$ ($m\ge2$);
\vskip 3pt
\item $k=-\frac{2}{3} h^{\vee}$ if  $\g=sp(n)\ (n\ge 6)$, or $\g=spo(2|m)\ (m\ge 2)$, or $\g=spo(n|m)\ (n\geq 4)$.
\end{enumerate}
\textbf{(II).} Assume that $\g^\natural=\g^\natural_0\oplus \g^\natural_1$ with $\g^\natural_0\simeq\C$ and $\g^\natural_1$ simple.

If $\g=sl(m|m-3)$ with $m\ge 4$, then there are no non--collapsing conformal levels.
 In  other cases the non--collapsing conformal levels are 
\begin{enumerate}
\item $k=-\frac{2}{3}h^\vee$ if $\g=sl(m|m+1)$ ($m\ge2$), and $\g=sl(m|m-1)$ ($m\ge 3$);
\vskip 2pt
\item $k=-\frac{2}{3}h^\vee$ and $k=-\frac{h^\vee-1}{2}$ in all other cases.
\end{enumerate}
\textbf{(III).} Assume that $\g^\natural=\sum_{i=1}^r\g^\natural_i$ with $\g^\natural_1\simeq sl(2)$ and $r\ge 2$.
If $\g=osp(n+5|n)$ with $n\ge 2$ or $\g=D(2,1;a)$ with $a=\half,-\half,-\frac{3}{2}$, then there are no non--collapsing conformal levels. 
In the other cases the non--collapsing conformal levels are 
 \begin{enumerate}
 \item $k=-\frac{h^\vee-1}{2}$ if $\g=D(2,1;a)$ ($a\not\in\{\half,-\half,-\frac{3}{2}\}$), $\g=osp(n+8|n)$ ($n\ge0$), $\g=osp(n+2|n)$ ($n\ge2$), $\g=osp(n-4|n)$ ($n\ge8$);
 \vskip 3pt
\item $k=-\frac{2}{3}h^\vee$ if $\g=osp(n+7|n)$ ($n\ge0$),
$\g=osp(n+1|n)$ ($n\ge4$);
\vskip 3pt
\item $k=-\frac{2}{3}h^\vee$ and $k=-\frac{h^\vee-1}{2}$ in all other cases.
 \end{enumerate}
\end{proposition}

It is important to observe that, if $k$ is a conformal level, we have the following identification of the Zhu algebra of $W_ k(\g, \theta)$.

\begin{proposition} \label{omotacka}
Assume that $k$ is a conformal non-collapsing level. Let $\mathcal J$ be  any proper  ideal in $W^k(\g, \theta)$ which contains $\omega-\omega_{sug}$. Then there is a surjective homomorphism of associative algebras 
$$  A( \mathcal V^k(\g ^{\natural}  ) )\rightarrow  A(W^ k(\g, \theta)  / \mathcal J ).  $$
In particular,  $A(W_k(\g, \theta) )$ is isomorphic to a certain quotient  of $U(\g ^{\natural})$.
\end{proposition}
\begin{proof}
Recall first that if a vertex algebra $V$ is strongly generated by the set $S \subset V$, then Zhu's algebra $A(V)$ is generated by the set $\{ [a], \ a \in S\}$ (cf. \cite[Proposition 2.5]{Abe}, \cite{DK}).  Since $ W^ k(\g, \theta)  / \mathcal J $ is strongly generated by the set 
$$ \{ G^{ \{ u \} } , \ u \in \g _{-1/2}  \} \cup \{ J ^{\{ x \} } , \ x\in \g^{\natural} \}, $$
we have that   $ A(W^ k(\g, \theta) / \mathcal J)$ is generated by the set 
$$ \{ [G^{ \{ u \} } ], \ u \in \g _{-1/2}  \} \cup \{ [J ^{\{ x \} }], \ x\in \g^{\natural} \}. $$
On the other hand, since 
$ G^{ \{ u \} } =  G^{ \{ u \} } \circ {\bf 1} \in O ( W^ k(\g, \theta)  / \mathcal J)$, we have $ [G^{ \{ u \} } ] = 0$ in $ A(W^ k(\g, \theta)  / \mathcal J )$  for every $u \in  \g _{-1/2}$. Therefore,
 $ A(W^ k(\g, \theta) / \mathcal J)$ is  only generated by the set 
$ \{ [J ^{\{ x \} }], \ x\in \g^{\natural} \}$. This gives a surjective homomorphism $A( \mathcal V^k(\g ^{\natural}  ) )  = U(\g ^{\natural}) \rightarrow A(W^ k(\g, \theta)  / \mathcal J )$. 
\end{proof}

We should also mention that a conjectural generalization of our results to  conformal embeddings of affine vertex algebras into  more general $W$--algebras have been recently proposed by T. Creutzig in \cite{TC}.

 \section{Some results on admissible affine vertex algebras}
Assume $\g$ is a simple Lie superalgebra. Let $\mathcal O^ k$ be the category of $\widehat{\g}$-modules from the category $\mathcal O$ of level $k$. 
Let $KL^ k$ be the subcategory of  $\mathcal O^ k$ consisting of modules  on which $\g$ acts locally finitely. Note that modules from  $KL^k$ are $V^k (\g)$-modules. Moreover, every irreducible module $M$ from $KL^ k$ has finite-dimensional weight spaces with respect to $(\omega^\g_{sug})_0$  and admits the following $\Z_{\ge 0}$--gradation:
$$ M= \bigoplus_{n \in \Z_{\ge 0} } M(n), \quad (\omega^\g_{sug})_0 \vert _{M(n)} \equiv (n + h) \mbox{Id} \quad (h \in {\C} ), $$
 (cf. \cite{KL}; such modules are usually called ordinary  modules in the the  terminology of  vertex operator algebra theory \cite{DLM-reg}).
 The graded component $M(0)$ is usually called the lowest graded component.

\subsection{ Fusion rules for certain affine vertex algebras }

 The classification of irreducible modules  in the category $\mathcal{O} ^k $ for affine vertex algebras $V_k(\g)$  at admissible levels was conjectured in \cite{AM1} and proved by  Arakawa  in \cite{Ar-rationality}. We need the classification result in the subcategory  $KL^ k$ of the category $\mathcal{O}^k $.
 
 \begin{defi} We define $KL_k$ to be the category of all modules $M$ in $KL^k$ which are 
 $V_k(\g)$-modules.
 \end{defi}

  The classification of irreducible modules in the category $KL_k$ coincides with the  classification  of irreducible  $V_k(\g)$-modules  having finite-dimensional weight spaces with respect to
 $(\omega^\g_{sug}) _0$ \cite{AM1}, \cite{Ar-rationality}.

 
 
 We restrict our attention to $\g=sl(n)$ with $(\cdot|\cdot)$ the trace form. We choose a  set of  positive roots for $\g$ and let $\omega_i\in\h^*$ ($i=1,\ldots,n-1$) denote the corresponding fundamental weights. Set $\Lambda_i=\Lambda_0+\omega_i$. Recall from \ref{fr}Ê the definition of fusion rules.
 
 \begin{proposition} \label{f-r-affine-1}\  Let  $k = \frac{1}{2}-n$, $n \ge 2$. \par
(1)   The set
 \bea \{  L_{sl(2n-2) }  (k\Lambda_0 + \Lambda_i) \ \vert \ i = 0, \dots, 2n-3\}  \label{scm-1} \eea
 provides a complete list of  irreducible $V_{k+1} (sl(2n-2))$-modules in the category $KL_{k+1}$.
\par
(2)   The following fusion rule holds in $KL_{k+1}$:
$$   L_{sl(2n-2)}  (k\Lambda_0 + \Lambda_{i_1} )  \times  L_{sl(2n-2) }  (k\Lambda_0 + \Lambda_{i_2})  =  L_{sl(2n-2)}  (k\Lambda_0 + \Lambda_{i_3}) $$
where $0 \le i_1, i_2, i_3 \le 2n-3 $ are  such that $i_1 + i_2 \equiv  \ i_3 \ \mbox{mod} (2n-2). $ 

In particular, the modules in \eqref{scm-1} are simple currents in the category $KL_{k+1}$.
 \end{proposition}
 \begin{proof}
First we notice that the set of admissible weights of level $k+1$ which are dominant with respect to $sl(2n-2)$ is 
$\{ k \Lambda_0 + \Lambda_i  \ \vert \ i = 0, \dots, 2n-3\}.$ Now assertion (1) follows from the main result from \cite{Ar-rationality}. \par
Assertion (2) follows from (1) and  the fact that the tensor product 
$$V_{sl(2n-2) } (\omega_{i_1} ) \otimes V_{sl(2n-2) }   (\omega_{i_2} ) $$ contains a component $ V_{sl(2n-2) }   (\omega_{i_3} ) $   if and only if $i_1 + i_2 \equiv  \ i_3 \ \mod (2n-2). $
 \end{proof}

 The proof of the following result is completely analogous to the proof of Proposition  \ref{f-r-affine-1}.

  \begin{proposition} \label{f-r-affine-2} \ Let $k = \frac{2}{3} (n-2) \notin {\Z}$. Then \par
(1) The set
 \bea  \{  L_{sl(n)}  (-(k+2) \Lambda_0 + \Lambda_i) \ \vert \ i = 0, \dots, n-1\}  \label{scm-2} \eea
 provides a complete list of  irreducible $V_{-k-1} (sl(n))$-modules in the category $KL_{-k-1}$.
\par (2)   The following fusion rules hold in $KL_{-k-1}$:
$$   L_{sl(n)}  (-(k+2) \Lambda_0 + \Lambda_{i_1} )  \times  L_{sl(n) }  (-(k+2) \Lambda_0 + \Lambda_{i_2})  =  L_{sl(n) }  (-(k+2) \Lambda_0 + \Lambda_{i_3}) $$
where $0 \le i_1, i_2, i_3 \le n-1 $ are such that $i_1 + i_2 \equiv  \ i_3 \ \mod (n). $

In particular, the modules in \eqref{scm-2} are simple currents in the category $KL_{-k-1}$.
 \end{proposition}

 \begin{rem}
 It is  also interesting to notice that the fusion algebra generated by   irreducible  modules for  $V_{3/2 -n } (sl(2n-2))$  in the category  $KL_{3/2-n }$ (resp.    for $V_{-\frac{2n-1}{3} } (sl(n))$ in the category $KL_{-\frac{2n-1}{3} }$)  is isomorphic to the fusion algebra for the rational affine vertex algebra $V_1 (sl(2n-2))$ (resp. $V_{1 } (sl(n))$). Moreover, all irreducible   modules in the $KL_k$ category   for these vertex algebras are simple currents. 
 \end{rem}

\subsection{The vertex algebra $V_{k} (sl(2))$ } 
Recall that a level $k$ is called admissible if $k\Lambda_0$ is abmissible. If $\g=sl(2)$ then $k$ is admissible if and only if $k+2 = \frac{p}{q}$, $p,q \in {\mathbb N}$, $(p,q) =1$, $p\ge 2$ \cite{KW-1988}. Let $e,h,f$ be the usual Chevalley generators for $sl(2)$.

\begin{theorem}  \label{unique-ideal} Assume that $k= \frac{p}{q}-2$ is an admissible level for $\widehat{sl_2}$. Then we have:
\par\noindent (1) \cite[Corollary 1]{KW-1988}. The maximal ideal in $J^k $ in $V^k (sl(2))$  is generated by a singular vector $v_{\lambda}$  of weight $  \lambda = (k- 2 (p-1) ) \Lambda_0 + 2 (p-1) \Lambda_1 $.
\par\noindent (2). The ideal $J^k  $ is  simple.
\end{theorem}

\begin{proof} We   provide here a proof of (2)  which uses Virasoro vertex  algebras and  Hamiltonian reduction.  This result can be also  proved by  using embedding diagrams for submodules of the  Verma modules for $\widehat{sl_2}$.

 Assume first  that $k \notin \Z_{\ge 0}$. Let $V^{Vir}(c_{p,q})$ be the universal Virasoro vertex algebra of central charge $c_{p,q} = 1 - 6 \frac{(p-q) ^2} {p q}$. Then the maximal ideal in $V^{Vir}(c_{p,q})$ is irreducible and it is generated by a  singular vector of conformal weight $(p-1) (q-1)$ (cf. \cite{GK}, Theorem 4.2.1). So $V^{Vir}(c_{p,q})$  contains  a unique   ideal which we shall denote by $I_{p,q}$.  Then $L^ {Vir} (c_{p,q} ) =  V^{Vir}(c_{p,q}) / I_{p,q}$ is a simple vertex algebra.
 
 Recall that  by quantum  Hamiltonian reduction
$$ W^k (sl(2),\theta) = V^{Vir}(c_{p,q}).$$
Let $H_{Vir}$ be the corresponding functor (denoted in \cite{Ar-2005}  by $H ^{\frac{\infty}{2} + 0 } _f  $), which maps  $V^k (sl(2))$-modules to  $V^{Vir}(c_{p,q})$-modules.
Assume that  $I$ is a non-trivial, proper  ideal  in $V^k (sl(2))$. By using  the main result of \cite{Ar-2005}, we get that $H_{Vir}( I) \ne 0, \ H_{Vir} (I) \ne V^{Vir} (c_{p,q}) $. So $H_{Vir} ( I) = I_{p,q}$. Since the functor $H_{Vir} $ is exact, we get that
$$H_{Vir} ( V^k (sl(2)) / I ) = V^{Vir}(c_{p,q}, 0) / I_{p,q} = L^ {Vir} (c_{p,q}). $$
By using again  the exactness and non-triviality result  of the functor $H_{Vir}$ we conclude that $V^k (sl(2)) / I$ is simple. So $I$ is the maximal ideal.

If $k \in \Z_{\ge 0}$, then the maximal ideal is $J^k= V^{k}(sl(2))\cdot (e_{(-1)}) ^{k+1}\vac $ and we have
$$ H_{Vir}(J^k) =W^k (sl(2),\theta) = V^ {Vir} (c_{k+2,1}) = L^ {Vir} (c_{k+2,1}).$$
Since $H_{Vir}(J^k)$ is irreducible, the properties of the functor $H_{Vir}$ imply  that $J^k$ is a simple ideal.
\end{proof}

It follows from \cite{GK}, Theorem 9.1.2, that, if $\g$ is a simple Lie algebra different from $sl(2)$, then the maximal ideal in $V^k(\g)$ is either zero or it is not simple.

\subsubsection{Representation theory of $V_{-1/2} (sl(2))$ } \label{sub-12}
We  now recall some known facts on the representation theory of the vertex algebra $V_{-1/2} (sl(2))$ (cf. \cite{AM1} and Theorem \ref{unique-ideal}).

We first fix notation. Let $\mathcal  L_{sl(2)} (\lambda)$ be  a highest weight $V^{-1/2}  (sl(2))$-module with highest weight   $\lambda $, and let $v_{\lambda}$ be the corresponding highest weight vector. 
  Writing $\lambda= -1/2\Lambda_0+\mu$ with $\mu\in\h^*$, we let $N_{sl(2)} (\lambda)$ denote  the generalized Verma module induced from the simple $sl(2)$-module  $V_{sl(2)} (\mu)$.\par  Let $\omega^{sl(2)}_{sug}$
be the Sugawara Virasoro vector for $V^{-1/2} (sl(2))$.
For $i \in \Z_{\ge 0}$ we define the following weights:
$$ \lambda_i = -(i+1/2) \Lambda_0 + i \Lambda_1= - 1/2 \Lambda_0 + i \omega_1.$$
Then one has:\par
\noindent(1). The maximal ideal of $V^ {-1/2} (sl(2))$  is generated by the singular vector $v_{\lambda_4} \in V^ {-1/2} (sl(2))$ of weight $\lambda_4$. In particular,
 $$ V_{-1/2} (sl(2)) = {V^ {-1/2} (sl(2))}\big/ V^ {-1/2} (sl(2))\cdot v_{\lambda_4}. $$
 Moreover $V^ {-1/2} (sl(2))\cdot v_{\lambda_4}$ is simple.
\vskip5pt
\noindent(2).  There is a singular vector $v_{\lambda_3} \in N_{sl(2)} (\lambda_1)$ of weight $\lambda_3$ such that  $$ L(\lambda_1)= {N_{sl(2)} (\lambda_1)}\big/{ V^ {-1/2} (sl(2))\cdot v_{\lambda_3}}. $$
Moreover ${ V^ {-1/2} (sl(2))\cdot v_{\lambda_3}}$ is simple.
\vskip5pt
\noindent(3). $L_{sl(2)} (\lambda_i)$, $i=0,1$,  are irreducible $V_{-1/2}(sl(2))$-modules.\newline\noindent Every $V_{-1/2}(sl(2))$-module  from the category $KL_{-\frac{1}{2}}$ is completely redu\-cible and   isomorphic to a direct sum of certain copies of $L_{sl(2)} (\lambda_i)$, $i=0,1$.
\vskip5pt
\noindent(4). The following fusion rule holds in $KL_{-\frac{1}{2}}$:
 \bea  L_{sl(2)} (\lambda_1) \times L_{sl(2)} (\lambda_1) = V_{-1/2} (sl(2)). \label{fr-12} \eea
 This fusion rule follows from the tensor product decomposition
 $$V_{sl(2)} (\omega_1) \otimes V_{sl(2)} (\omega_1) = V_{sl(2)} (2 \omega _1) + V_{sl(2)} (0)$$
 and the classification of irreducible modules for $ V_{-1/2} (sl(2))$-modules from \cite{AM1}). In particular, we only need to note that  $L_{sl(2)} (\lambda_2)$ is not a $V_{-1/2} (sl(2))$-module.

\section{Semisimplicity of conformal embeddings}\label{examples}
The main goal of this section is to give criteria for establishing  the simplicity of $\mathcal V_k(\g^\natural)$ together with the semisimplicity of $W_k(\g,\theta)$ as a $\mathcal V_k(\g^\natural)$-module when $k$ is a non-collapsing conformal level. We will give two separate criteria: one for the cases when $\g^\natural$ has a nontrivial center and another for the cases when $\g^\natural$ is centerless.

\subsection{Semisimplicity with nontrivial center of $\g^\natural$}\label{centersemisimple}The next result collects some structural facts proven in \cite[Proposition 4.6]{AKMPP1} describing the structure of  $\g_{-1/2}$ as a $\g^\natural$-module.
\begin{lemma}\label{structureg12}Assume that $\g^\natural$ is a Lie algebra and $\g^\natural_0\ne\{0\}$ (which happens only for $\g=sl(n)$ or $\g=sl(2|n)$, $n\ne 2$). Then
\begin{enumerate}
\item $\Dim\,\g^\natural_0=1|0$.
\item A basis $\{c\}$ of $\g^\natural_0$ can be chosen so that the eigenvalues of $ad(c)$ acting on $\g_{-1/2}$ are $\pm1$.
\item Let $U^+$ (resp. $U^-$) be the eigenspace for $ad(c)_{|\g_{-1/2}}$ corresponding to the eigenvalue $1$ (resp. $-1$). Then $\g_{-1/2}=U^+\oplus U^-$ with $U^\pm$ irreducible finite dimensional mutually contragredient $\g^\natural$-modules.
\end{enumerate}
\end{lemma}

By \eqref{JG} and the above Lemma, $J^{ \{c  \} } _{(0)}$ defines a $\Z$-gradation on ${W}_{k}(\g, \theta)$: 
$$ {W}_{k}(\g, \theta)= \bigoplus {W}_{k}(\g, \theta)^{(i) } , \quad {W}_{k}(\g, \theta)^{(i) } = \{ v \in {W}_{k}(\g, \theta)\ \vert \ J^{ \{ c \} } _{(0)} v = i v \}. $$

Recall that a primitive vector in a module $M$ for an affine vertex algebra is a vector that is singular in some subquotient of $M$.

In light of Lemma \ref{structureg12}, we have that, in the Grothendieck group of finite dimensional representations of $\g^\natural$, we can write
$$
U^+\otimes U^-=V(0)+ \sum_{\nu_i\ne 0} V(\nu_i).
$$

\begin{theorem}\label{semicenter}
Assume  that the embedding of $\mathcal V_k(\g^\natural)$ in ${W}_{k}(\g, \theta)$ is conformal and that
${W}_k(\g)^{(0)}$ does not contain $
\mathcal{V}_k(\g^\natural)$-primitive vectors of weight $\nu_{r}$.

Then ${W}_k(\g)^{(0)}=\mathcal{ V}_k(\g^\natural)$,  $\mathcal V_k(\g^\natural)$ is a simple affine vertex algebra and ${W}_{k}(\g, \theta)^{(i)}$ are simple  $\mathcal V_k(\g^\natural)$-modules. 
\end{theorem}
\begin{proof}Let $\mathcal U^\pm=span(G^{\{u\}}\mid u\in U^\pm)$. 
Let $ A^{\pm} =\mathcal  V_k(\g^\natural) \cdot\mathcal U^{\pm }$. 
We claim  that 
\begin{equation}\label{fusionproduct}
A^- \cdot A^+ \subset \mathcal  V_k(\g^\natural).
\end{equation}
To check this, it is enough to check for all $n\in\ganz$, $u\in U^+$, and $v\in U^-$, that ${G^{\{u\}}}_{(n)}G^{\{v\}}\in \mathcal  V_k(\g^\natural)$. Assume that this is not the case. Then we can choose $n$ maximal such that there are $u\in U^+$, $v\in U^-$ such that ${G^{\{u\}}}_{(n)}G^{\{v\}}\notin \mathcal  V_k(\g^\natural)$. Since the map 
$$
\phi: U^+\otimes U^-\to {W}_k(\g)/\mathcal V_k(\g^\natural),\quad \phi:u\otimes v \mapsto {G^{\{u\}}}_{(n)}G^{\{v\}}+\mathcal V_k(\g^\natural)
$$
is $\g^\natural$-equivariant, we can choose a weight vector $w=\sum u_i\otimes v_i\in U^+\otimes U^-$ of weight $\nu$ such that $\phi(w)$ is a highest weight vector in $\phi(U^+\otimes U^-)$. Since, by maximality of $n$, $\phi(w)$ is singular for $\mathcal V_k(\g^\natural)$, we have that $y=\sum {G^{\{u_i\}}}_{(n)}G^{\{v_i\}}$ is primitive in ${W}_{k}(\g, \theta)$ so, by our hypothesis,  $\nu=0$. Since the embedding is conformal, $y$ is an eigenvector for $\omega'_{sug}$ and since $\phi(w)$ is singular for $\mathcal V_k(\g^\natural)$ of weight $0$, we see that its eigenvalue is zero. Since the embedding is conformal we have that $y$ has conformal weight zero in ${W}_{k}(\g, \theta)$ so $y\in \C\vac\subset \mathcal V_k(\g^\natural)$, a contradiction.

Since the embedding is conformal, $ {W}_{k}(\g, \theta) $ is strongly generated by $$span\left\{J^{\{a\}} \mid\ a \in \g^\natural\right\} + \mathcal{U}^+ + \mathcal{U}^-.$$
It follows that  ${W}_{k}(\g, \theta)^{(0)}$ is contained in the sum of all fusion products 
of type $A_1\cdot A_2\cdot\ldots\cdot A_r$ with $A_i\in\{A^+,A^-,\mathcal V_k(\g^\natural)\}$ such that
$$
\sharp\{i\mid A_i=A^+\}=\sharp\{i\mid A_i=A^-\}.
$$
 By the associativity  of fusion products, we see that \eqref{fusionproduct} implies that $A_1\cdot A_2\cdot\ldots\cdot A_r\subset\mathcal V_k(\g^\natural)$, so ${W}_{k}(\g, \theta)^{(0)}=\mathcal V_k(\g^\natural)$. 
It follows that $\mathcal V_k(\g^\natural)$  is a simple affine vertex algebra  and ${W}_{k}(\g, \theta)^{(i)}$ is a simple $\mathcal V_{k}(\g^\natural)$-module for all $i$.
\end{proof}

If $V(\mu)$, $\mu\in (\h^\natural)^*$, is an irreducible  $\g^\natural$-module,  then we can write
\begin{equation}\label{mui}
V(\mu)=\bigotimes_{j\in I}V_{\g^\natural_j}(\mu^{j}),
\end{equation}
 where $V_{\g^\natural_j}(\mu^{j})$ is an  irreducible $\g_j^\natural$-module. Let $\rho_0^j$ be the Weyl vector in $\g^\natural_j$ (with respect to the positive system induced by the choice of positive roots for $\g$).

\begin{cor}\label{conditionsfinite}
If the embedding of $\mathcal V_k(\g^\natural)$ in ${W}_{k}(\g, \theta)$ is conformal and, for each irreducible subquotient  $V(\mu)$ with $\mu\ne0$ 
of the $\g^\natural$-module $U^+\otimes U^-$, we have 
\bea\label{finioths}
 \sum_{i\in I,k_i\ne0}\frac{ ( \mu^{i} , \mu^{i} + 2 \rho_{0}^i ) }{ 2 (k_i+ h^\vee _{0,i})}\not\in \ganz_+,
\eea
 then ${W}_k(\g)^{(0)}=\mathcal{ V}_k(\g^\natural)$,  $\mathcal V_k(\g^\natural)$ is a simple affine vertex algebra and  the ${W}_{k}(\g, \theta)^{(i)}$ are simple  $\mathcal V_k(\g^\natural)$-modules. 
 \end{cor}
\begin{proof}
In order to apply Theorem \ref{semicenter}, we need to check that if $\mu\ne0$ and $V(\mu)$ is an irreducible subquotient of $U^+\otimes U^-$, then there is no primitive vector $v$ in ${W}_{k}(\g, \theta)^{(0)}$ with weight $\mu$. Since the embedding is conformal, $\omega'_{sug}$ acts diagonally on ${W}_{k}(\g, \theta)$. In particular, we can assume that $v$ is an eigenvector for $\omega'_{sug}$. Let $N\subset M\subset {W}_{k}(\g, \theta)$  be submodules such that $v+N$ is a singular vector in $M/N$. Then $v+N$ is an eigenvector for the action of $\omega'_{sug}$ on $M/N$ and the corresponding eigenvalue is $\sum_{i=0,k_i\ne0}^r\frac{ ( \mu^{i} , \mu^{i} + 2 \rho_{0}^i ) }{ 2 (k_i+ h^\vee _{0,i})}$. It follows that the eigenvalue for  $\omega'_{sug}$ acting on $v$ is $\sum_{i=0,k_i\ne0}^r\frac{ ( \mu^{i} , \mu^{i} + 2 \rho_{0}^i ) }{ 2 (k_i+ h^\vee _{0,i})}$.  It is easy to check that the conformal weights of elements in  ${W}_{k}(\g, \theta)^{(0)}$ are positive integers hence, since $\omega'_{sug}=\omega$, $\sum_{i=0,k_i\ne0}^r\frac{ ( \mu^{i} , \mu^{i} + 2 \rho_{0}^i ) }{ 2 (k_i+ h^\vee _{0,i})}$ must be in $\ganz_+$, a contradiction. 
\end{proof}

We now apply  Corollary \ref{conditionsfinite} to the cases where $\g^\natural$ is a basic  Lie superalgebra  with nontrivial center. These can be read off from Tables 1--3 and correspond to taking $\g=sl(n)$ ($n\ge3$),  $\g=sl(2|n)$ ($n\ge1$, $n\ne2$) or $\g=sl(m|n)$ ($n\ne m>2$).
 
 \begin{theorem} \label{finite-dec} Assume that we are in the following  cases of conformal embedding of  $ \mathcal V_k(\g^\natural) $ into ${W}_{k} (\g,\theta) $.
\begin{enumerate}
\item \label{his}${\mathfrak g} =sl(n) $, $n \ge 4$, conformal level $k=-\frac{n-1}{2}=-\frac{h^\vee-1}{2}$;\label{A-emb1}
\item \label{hiis}
${\mathfrak g} =sl(n)$, $n \ge 5$, $n \ne 6$, conformal level $k=-\frac{2n}{3}=-\frac{2h^\vee}{3}$;\label{A-emb2}
\item \label{hiiis}${\mathfrak g}=sl(2|n) $ , $n \ge 4$,  
conformal level $k=\frac{n-1}{2}=-\frac{h^\vee-1}{2}$;  \label{sl(2|n)-emb1}
\item \label{hivs}${\mathfrak g}=sl(2|n) $, $n \ge 3$,  
conformal level $k=\frac{2(n-2)}{3}=-\frac{2h^\vee}{3}$.  \label{sl(2|n)-emb2}
\item${\mathfrak g}=sl(m|n) $, $m>2$,  $m\ne n+3,n+2, n, n-1$,
conformal level $k=\frac{n-m+1}{2}=-\frac{h^\vee-1}{2}$;  \label{sl(2|n)-emb3}
\item${\mathfrak g}=sl(m|n) $, $ m>2$, $m\ne n+6,n+4,n+2,n$, 
conformal level $k=\frac{2(n-m)}{3}=-\frac{2h^\vee}{3}$.  \label{sl(2|n)-emb4}

\end{enumerate}
Then $\mathcal V_{k}(\g^\natural)$ is a simple affine vertex algebra and  $ {W}_{k}(\g, \theta) ^{(i) } $ is an irreducible  $\mathcal V_{k}(\g^\natural)$-module   for every $i \in {\Z}$. In particular,  ${W}_{k}(\g, \theta)$ is a semisimple $\mathcal V_{k}(\g^\natural)$-module.
\end{theorem}
\begin{proof} We verify that the assumptions of Corollary  \ref{conditionsfinite} hold.
In cases \eqref{A-emb1} and \eqref{A-emb2}, $\g^\natural\simeq gl(n-2)=\C Id\oplus sl(n-2)$, hence $\g^\natural_0\simeq\C $ and $\g^\natural_1\simeq sl(n-2)$. Moreover $U^{+} =\C^{n}$ and  $U^{-}=(\C^{n})^*$. The tensor product $U^+  \otimes U^-$ decomposes as $V(0)  \oplus V(\mu) $ with 
$$ \mu^0=0,\quad \mu^1 =\omega_1 + \omega_{n-3}.$$
Moreover $k_0=k+h^\vee/2\ne 0$ and $k_1=k+1\ne 0$. Since
$$
 \sum_{i=0,k_i\ne0}^r\frac{ ( \mu^{i} , \mu^{i} + 2 \rho_{0}^i ) }{ 2 (k_i+ h^\vee _{0,i})}= \sum_{i=0}^1\frac{ ( \mu^{i} , \mu^{i} + 2 \rho_{0}^i ) }{ 2 (k_i+ h^\vee _{0,i})}=\frac{n-2}{n+k-1},
 $$ we see that  \eqref{finioths} holds in cases \eqref{A-emb1} and \eqref{A-emb2}. 
 
 In cases \eqref{sl(2|n)-emb1} and \eqref{sl(2|n)-emb2}, $\g^\natural\simeq gl(n)=\C Id\oplus sl(n)$, hence $\g^\natural_0\simeq\C$ and $\g^\natural_1=sl(n)$. Moreover $U^{+} =\C^{n}$ and  $U^{-}=(\C^{n})^*$. The tensor product $U^+  \otimes U^-$ decomposes as $V(0)  \oplus V(\mu) $ with 
$$ \mu^0=0,\quad \mu^1 =\omega_1 + \omega_{n-1}.$$
Moreover $k_0=k+h^\vee/2\ne 0$ and $k_1=k+1\ne 0$. Since
$$
 \sum_{i=0,k_i\ne0}^r\frac{ ( \mu^{i} , \mu^{i} + 2 \rho_{0}^i ) }{ 2 (k_i+ h^\vee _{0,i})}= \sum_{i=0}^1\frac{ ( \mu^{i} , \mu^{i} + 2 \rho_{0}^i ) }{ 2 (k_i+ h^\vee _{0,i})}=\frac{n}{n+k},
 $$ we see that  \eqref{finioths} holds in cases \eqref{sl(2|n)-emb1} and \eqref{sl(2|n)-emb2}. \par
In  cases \eqref{sl(2|n)-emb3} and \eqref{sl(2|n)-emb4} we have $\g^\natural\simeq gl(m-2|n)=\C Id\oplus sl(m-2|n)$ (recall that we are assuming $m\ne n+2$), hence $\g^\natural_0\simeq\C$ and $\g^\natural_1=sl(m-2|n)$. 
Moreover $U^{+} =\C^{m-2|n}$ and  $U^{-}=(\C^{m-n|n})^*$. Then, as $\g^\natural$-modules
$$U^+\otimes U^- \simeq  sl(m-2|n)\oplus \C.$$
With notation as in \cite{GKMP}, choose $\{\e_1-\d_1,\ \d_1-\d_2,\ldots,\d_{n-1}-\d_n,\d_n-\e_2,\ldots,$ $\e_{m-1}-\e_m\}$ as simple roots for $\g$, so that the highest root is even.
The set of posi\-tive roots induced on  $\g^\natural$ has as simple roots
$\{
\d_1-\d_2,\ldots,\d_{n-1}-\d_n,\d_n-\e_2,\ldots,\e_{m-2}-\e_{m-1}
\}.$
Then we have $\mu_0=0, \mu_1=\d_1-\e_{m-1}$. 
Since 
$2(\rho_0^1)_1= -n(\e_2+\ldots+\e_{m-1})+(m-2)(\d_1+\ldots+\d_n),$ $2(\rho_0^1)_0= (m-3)\e_2+(m-5)\e_3+\cdots+(3-m)\e_{m-1}+(n-1)\d_1+(n-3)\d_2+\cdots+(1-n)\d_n$, we have that 
$$
(\d_1-\e_{m-1},2\rho_0^1)=-n+1+m+2-(3-m)-n=2(m-n-2).
$$

\vskip5pt
In  case \eqref{sl(2|n)-emb3}, we have  $k=-\frac{h^\vee-1}{2}$. Then 
$k_1+h_{0,1}^\vee=\frac{m-n-1}{2}$
and $(\mu,\mu+2\rho^1_0)=(\d_1-\e_{m-1},2\rho^1_0)=2(m-n-2)$. Therefore 
$$\frac{(\mu_1,\mu_1+2\rho^1_0)}{2(k_1+h_{0,1}^\vee)}=2\frac{m-n-2}{m-n-1}=2(1-\frac{1}{m-n-1})$$
which is not an integer unless $m=n+3,n+2,n,n-1$. 

In  case \eqref{sl(2|n)-emb4}, we have $k=-\frac{2}{3}h^\vee$. Then
$k_1+h_{0,1}^\vee=\frac{m-n-3}{3}$ and 
\vskip5pt
$$\frac{(\mu_1,\mu_1+2\rho^1_0)}{2(k_1+h_{0,1}^\vee)}=3\frac{m-n-2}{m-n-3}=3(1+\frac{1}{m-n-3})$$
which not an integer unless $m=n+6,n+4,n+2,n$.
\end{proof}
In Section \ref{explicit-decomposition} we will discuss explicit decompositions for some occurrences of  cases \eqref{his} and  \eqref{hivs} of Theorem \ref{finite-dec},  exploiting the fact that some of the levels $k_i$ may be  admissible for $V_{k_i}(\g_i^\natural)$.   We shall determine  explicitly  the decomposition of $ {W}_{k} (\g,\theta)$ as a module for this admissible vertex algebra.
\vskip10pt

We now list the  cases which are not covered by  Theorem \ref{finite-dec}. Recall that, if $\g=sl(m|n)$, then we excluded the case $m=n+2$ from the beginning while the case $m=n$ had to be excluded because $\g=sl(n|n)$ is not simple. The remaining cases are
\begin{enumerate}
\item $sl(n-1|n)$, $k=1$;
\item $sl(n+3|n)$ $n\ge0$, $k=-1$;
\item $sl(n+4|n)$, $k=-\frac{8}{3}$;
\item $sl(n+6|n)$, $k=-4$;
\item $sl(2|1)=spo(2|2)$, $k=-\frac{2}{3}$.
\end{enumerate}

If $\g=sl(n-1|n)$ and $k=1$ then $k+h^\vee=0$, so we have to exclude this case.

By  Theorem 3.3 of \cite{AKMPP1}
  (stated in this  paper as Proposition  \ref{class-coll}), $k=-1$ is a collapsing level for $\g=sl(n+3|n)$. It follows that
$W_{-1}(sl(n+3|n))=\mathcal V_{-1}(gl(n+1|n))$. If $n=0$, we obtain $W_{-1}(sl(3))=\mathcal V_{-1}(gl(1))$, which is the Heisenberg vertex algebra $V_{\frac{1}{2}}(\C c)$.

In the case $\g = sl(2|1)$, $k=-\frac{2}{3}$, ${W}_{k}(\g, \theta)$ is the simple $N=2$ superconformal vertex algebra  $V^{N=2} _{\mathbf c} $ (cf. \cite{KacV}, \cite{KW},  \cite{A-1999}) with central charge $\mathbf c=1$. In this case $\mathcal V_k(gl(1))$  is the Heisenberg vertex algebra $M(-\frac{2}{3})$, so we have conformal embedding of $M(-\frac{2}{3})$ into ${W}_{k}(\g, \theta)$. It is well-known that $V^{N=2} _{\mathbf c} $  admits the free-field realization as the lattice vertex algebra  $F_3$.
Using this realization we see that each $ {W}_{k}(\g, \theta) ^{(i) } $ is an irreducible  $M(-\frac{2}{3})$-module.
\par
Case (3) with $n=0$ is of special interest: it turns out that ${W}_{k}(\g, \theta)$ is isomorphic to the vertex algebra $\mathcal R ^{(3)}$ introduced in \cite{A-2014}.
 This will require a thoughtful discussion which will be performed in Section \ref{sect-r3}. There we will prove the following 
\begin{theorem}\label{A3special} Let $\g$ be of type $A_3$ with $k=-\frac{8}{3}$.
Then, for all $i\in\ganz$, $ {W}_{k}(\g, \theta)^{(i)}$ is an infinite sum of irreducible $\mathcal V_{k}(gl(2))$-modules. 
\end{theorem}


\begin{rem}  The only remaining open case is $\g$ of type $A_5$ with $k=-4$.
It is surprising that  this case (possibly) and  the case discussed in Theorem \ref{A3special}   are  the only  instances  of conformal embeddings into $ {W}_{k}(\g, \theta) $  where $ {W}_{k}(\g, \theta)^{(i)}$ is not a finite sum of irreducible $\mathcal V_{k}(\g^\natural)$-modules.
\end{rem}

 \subsection{Finite decomposition when $\g^\natural$ has trivial center}
 
 Recall that ${W}_{k} (\g,\theta)$ is a $\frac{1}{2}\Z_{\ge 0}$-graded vertex algebra by conformal weight. It admits the following natural $\Z_2$-gradation
 $$ {W}_{k} (\g,\theta) = {W}_{k} (\g,\theta) ^ {\bar 0 }  \oplus  {W}_{k} (\g,\theta)^ {\bar 1},$$
 where 
$$ {W}_{k} (\g,\theta) ^ {\bar i  } = \{ v \in {W}_{k} (\g,\theta) \mid\Delta_v \in i/2 + \Z \}.$$

Similarly to what we have done in Section \ref{centersemisimple}, we start by developing a criterion for checking when $\mathcal V_{k} ({\mathfrak g}^{\natural})={W}_{k} (\g,\theta) ^ {\bar0 } $, so that $\mathcal V_{k} ({\mathfrak g}^{\natural})$ is a simple affine vertex algebra, and ${W}_{k} (\g,\theta) ^ {\bar1} $ is an irreducible $\mathcal V_{k} ({\mathfrak g}^{\natural})$-module. In particular, in these cases, we have a  finite decomposition of ${W}_{k} (\g,\theta) $ as a $\mathcal V_{k} ({\mathfrak g}^{\natural})$-module.

One checks, browsing Tables 1--3, that when $\g^\natural$ has trivial center,  then $\g_{-1/2}$ is an  irreducible  $\g^\natural$-module. 
Then, in the Grothen\-dieck group of finite dimensional representations of $\g^\natural$, we can write
$$
\g_{-1/2}\otimes \g_{-1/2}=V(0)+ \sum_{\nu_i\ne 0} V(\nu_i).
$$

\begin{theorem}\label{seminocenter}
Assume  that the embedding of $\mathcal V_k(\g^\natural)$ in ${W}_{k}(\g, \theta)$ is conformal and that
${W}_k(\g)^{\bar0}$ does not contain $
\mathcal{V}_k(\g^\natural)$-primitive vectors of weight $\nu_{i}$.

Then ${W}_k(\g)^{\bar0}=\mathcal{ V}_k(\g^\natural)$,  $\mathcal V_k(\g^\natural)$ is a simple affine vertex algebra, and ${W}_{k}(\g, \theta)^{\bar1}$ is a simple  $\mathcal V_k(\g^\natural)$-module. 
\end{theorem}
\begin{proof}Let $\mathcal U=span\{G^{\{u\}}\mid u\in \g_{-1/2}\}$. 
Let $ A =\mathcal  V_k(\g^\natural)\cdot \mathcal U$. 
Arguing as in the proof of Theorem \ref{semicenter}, we have
\begin{equation}\label{fusionproduct1}
A \cdot A \subset \mathcal  V_k(\g^\natural).
\end{equation}

From \eqref{fusionproduct1} we obtain that $\mathcal  V_k(\g^\natural)+A$ is a vertex subalgebra of ${W}_{k}(\g, \theta)$. 
Since the embedding is conformal, $ {W}_{k}(\g, \theta) $ is strongly generated by $$span(J^{\{a\}} \mid\ a \in \g^\natural ) + \mathcal{U},$$
hence $\mathcal  V_k(\g^\natural)+A$ contains a set of generators for $ {W}_{k}(\g, \theta) $. It follows that 
$${W}_{k}(\g, \theta)=\mathcal  V_k(\g^\natural)+A,\ {W}_k(\g,\th)^{\bar0}=\mathcal  V_k(\g^\natural),\ {W}_k(\g,\th)^{\bar1}=\mathcal  V_k(\g^\natural)\cdot\mathcal U. 
$$

The statement now follows by a simple application of  quantum Galois theory, for the cyclic group of order 2.
\end{proof}

The same argument of Corollary \ref{conditionsfinite} provides a numerical criterion for a finite decomposition, actually, for a decomposition in a sum of two irreducible submodules:

\begin{cor}\label{conditionsfinite1}
If the embedding of $\mathcal V_k(\g^\natural)$ in ${W}_{k}(\g, \theta)$ is conformal and, for any irreducible subquotient  $V(\mu)$ of $\g_{-1/2}\otimes\g_{-1/2}$ with $\mu\ne0$, we have (see \eqref{mui})
\bea\label{finitedec}
 \sum_{i=0,k_i\ne0}^r\frac{ ( \mu^{i} , \mu^{i} + 2 \rho_{0}^i ) }{ 2 (k_i+ h^\vee _{0,i})}\not\in \ganz_+,
\eea
 then ${W}_k(\g)^{\bar0}=\mathcal{ V}_k(\g^\natural)$,  $\mathcal V_k(\g^\natural)$ is a simple affine vertex algebra and ${W}_{k}(\g, \theta)^{\bar1}$ is an irreducible  $\mathcal V_k(\g^\natural)$-module. 
 \end{cor}

We now apply  Corollary \ref{conditionsfinite1} to the cases where $\g^\natural$ is a basic  Lie superalgebra. These can be read off from Tables 1--2 and correspond to taking $\g=so(n)$ ($n\ge7$), $\g=sp(n)$ ($n\ge4$), $\g=psl(2|2)$,  $\g=spo(2|m)$ ($m\ge3$), $\g=osp(4|m)$ ($m\ge1$), $\g=psl(m|m)$ ($m>2$), $\g=spo(m|n)$ ($n\ge 4$), $\g=osp(m|n)$ ($m\geq 5)$ or $\g$ of the following exceptional types: $G_2$, $F_4$, $E_6$, $E_7$, $E_8$, $F(4)$, $G(3)$, $D(2,1;a)$. 

\begin{theorem} \label{finite-dec-shorter}
Assume that we are in the following  cases of conformal embedding of  $  \mathcal V_{k} ({\mathfrak g}^{\natural})$  into ${W}_{k} (\g,\theta) $:
\begin{enumerate}
\item ${\mathfrak g}=so(n) $  ($n \ge 8$, $n\ne 11$),   $\g=osp(4|n)$  ($n \ge 2$), $\g=osp(m|n)$ ($m\geq 5$, $m\ne n+r$, $r\in\{-1,2,3,4,6,7,8,11\}$) or $\g$ is of type $G_2$, $F_4$ $E_6$, $E_7$, $E_8$, $ F(4)$  and $k =-\frac{h^{\vee} -1}{2}$.
\item ${\mathfrak g}=sp(n) $($n \ge 6$), $\g=spo(2|m)$ ($m \ge 3$, $m\ne4$),  $\g=spo(n|m)$ ($n\geq 4$, $m\ne n+2,n,n-2,n-4$) and $k=-\tfrac{2}{3} h^{\vee}$.
\end{enumerate}
Then ${W}_{k} (\g,\theta) ^ {\bar i }$, $i=0,1$, are irreducible $\mathcal V_{k} ({\mathfrak g}^{\natural})$-modules
\end{theorem}

\begin{proof}

We shall show  case-by-case that the numerical criterion of Corollary  \ref{conditionsfinite1} holds. We start by listing all cases explicitly.

\begin{enumerate}
\item ${\mathfrak g} $ is of type $D_{n}$, $n \ge 5$,  
$k =-\frac{h^{\vee} -1}{2} = \tfrac{3}{2} -n$;  \label{D-emb}
\item ${\mathfrak g} $ is of type $B_n$, $n \geq 4$, $n \neq 5$, 
$k= -\frac{h^{\vee} -1}{2} =1-n$; \label{B-emb}
\item ${\mathfrak g} $ is of type $G_2$, 
 $k=-\frac{h^{\vee} -1}{2} =-3/2$;  \label{G2}
\item ${\mathfrak g} $ is of type $F_4$, 
 $k=-\frac{h^{\vee} -1}{2} =-4$;  \label{F4}
\item ${\mathfrak g} $ is of type $E_6$,  
$k =-\frac{h^{\vee} -1}{2} =-11/2$;\label{E6}
\item ${\mathfrak g} $ is of type $E_7$, 
$k =-\frac{h^{\vee} -1}{2} =- 17/2$; \label{E7}
 \item ${\mathfrak g} $ is of type $E_8$, 
$k =-\frac{h^{\vee} -1}{2} =-29/2$; \label{E8}
 \item ${\mathfrak g} =F(4) $, 
$k=-\frac{h^{\vee}-1}{2} = 3/2$; \label{F(4) }
\item ${\mathfrak g} =osp(4| 2n)$, $n\ge2$,
$k=-\frac{h^{\vee}-1}{2} =n - 1/2$; \label{osp(4,2n)}
\item ${\mathfrak g} $ is of type $C_{n+1}$, $n \ge 2$,  
$k =-\tfrac{2}{3} h^{\vee}=-\tfrac{2}{3} (n+2)$;  \label{C-emb}
\item ${\mathfrak g} =spo(2|2n)$, $n \ge 3$,  
 $k=-\frac{2}{3} h^{\vee}=\frac{2}{3} (n-2)$;  
\label{spo(2,2n)}
\item ${\mathfrak g} =spo(2|2n+1)$, $n \ge 1$,  
 $k=-\frac{2}{3} h^{\vee}=\frac{2}{3} (n-3/2) $; \label{spo(2,2n+1)}
 \item\label{spo1} ${\mathfrak g} =spo(n|m)$ $n\ge4$, $k=-\frac{2}{3}h^\vee=\frac{m-n-2}{3}$;
    \item\label{spo3} ${\mathfrak g} =osp(m|n)$ $m\ge5$, $k=-\frac{h^\vee-1}{2}=\frac{n-m+3}{2}$.
 \end{enumerate}
If $V(\mu)$ is an irreducible representation of $\g^\natural$, we set 
$$h_\mu= \sum_{i=0,k_i\ne0}^r\frac{ ( \mu^{i} , \mu^{i} + 2 \rho_{0}^i ) }{ 2 (k_i+ h^\vee _{0,i})}.
$$
 For each case listed above we give $\g^\natural$, $\g_{-1/2}$, and the decomposition of $\g_{-1/2}\otimes \g_{-1/2}$ in irreducible modules for $\g^\natural$. Then we list all values $h_\mu$ for all irreducible components $V(\mu)$ of $\g_{-1/2}\otimes \g_{-1/2}$ with $\mu\ne0$, showing that they are not positive integers. We  also exhibit the decomposition of ${W}_{k}(\g, \theta)$ as a $\mathcal V_k(\g^\natural)$-module. In cases \eqref{spo1}--\eqref{spo3} we will use the usual $\e-\d$ notation for roots in Lie superalgebras explained e.g. in \cite{GKMP}.
\vskip 5pt
\item{ Case \eqref{D-emb}: } \label{D_n} $\g^\natural$ of Type $A_1\times D_{n-2}$, $\g_{-1/2}=V_{A_1}( \omega_1)\otimes V_{D_{n-2}}(\omega_1)$
%
%
\begin{align*}
\g&_{-1/2}\otimes \g_{-1/2}\\&= (V_{A_1}(2  \omega_1) + V_{A_1}( 0))\otimes(V_{D_{n-2}}(0 ) + V_{D_{n-2}}( \omega_2 ) + V_{D_{n-2}}( 2 \omega_1)).
\end{align*}
%
%
Values of $h_\mu$'s:
$$ h_{2 \omega_1, 0} =  \frac{4}{3} ,  \quad h_{0, \omega_2} = \frac{4n-12}{2n-5}, \  h_{0, 2\omega_1} = \frac{4n-8}{2n-5},$$
$$h_{2 \omega_1, \omega_2} = \frac{4}{3} + \frac{4n-12}{2n-5} ,  \quad h_{2 \omega_1, 2 \omega_1} =\frac{4}{3} + \frac{4n-8}{2n-5}.$$
 These values are   non-integral for $n \ge 5$.
 
Decomposition:
\begin{align*}{W}_{k} (D_n) = L_{A_1}(-\tfrac{1}{2} \Lambda_0) &\otimes L_{D_{n-2} }(( \tfrac{7}{2} -n ) \Lambda_0) \\&\oplus L_{A_1}(-\tfrac{3}{2} \Lambda_0 + \Lambda_1) \otimes L_{D_{n-2} }(( \tfrac{5}{2} -n ) \Lambda_0+ \Lambda_1) . 
\end{align*}

\item{ Case \eqref{B-emb}: }$\g^\natural$ of Type $A_1\times B_{n-2}$, $\g_{-1/2}=V_{A_1}( \omega_1)\otimes V_{B_{n-2}}(\omega_1)$
%
%
%
%
\begin{align*}
\g&_{-1/2}\otimes \g_{-1/2}\\&= (V_{A_1}(2  \omega_1) + V_{A_1}( 0))\otimes(V_{B_{n-2}}(0 ) + V_{B_{n-2}}( \omega_2 ) + V_{B_{n-2}}( 2 \omega_1)).
\end{align*}
Values of $h_\mu$'s:
%
$$ h_{2 \omega_1, 0} =  \frac{4}{3} ,  \quad h_{0, \omega_2} = \frac{2n-5}{n-2}, \  h_{0, 2\omega_1} = 
\frac{2n-3}{n-2},$$
$$h_{2 \omega_1, \omega_2} = \frac{4}{3} + \frac{2n-5}{n-2} ,  \quad h_{2 \omega_1, 2 \omega_1} =\frac{4}{3} + \frac{2n-3}{n-2}.$$
These values are   non-integral for $n \ge 4$, $n \neq 5$.
%

Decomposition:
\begin{align*}{W}_{k} (B_n) &= L_{A_1}(-\tfrac{1}{2} \Lambda_0) \otimes L_{B_{n-2} }(( 3 -n ) \Lambda_0) \\
&\oplus L_{A_1}(-\tfrac{3}{2} \Lambda_0 + \Lambda_1) \otimes L_{B_{n-2} }(( 2 -n ) \Lambda_0+ \Lambda_1) . 
\end{align*}

\item{Case \eqref{G2}: }$\g^\natural$ of Type $A_1$, $\g_{-1/2}=V_{A_1}(3\omega_1)$,
$$\g_{-1/2}\otimes \g_{-1/2} = V_{A_1}(6  \omega_1) + V_{A_1}(4  \omega_1) + V_{A_1}( 2 \omega_1) +V_{A_1}( 0)  . $$
Values of $h_\mu$'s:
$$ h_{ 2 i\omega_1}=\frac{2}{5}i(i+1) \notin {\Z} \quad (i=1,2,3).$$
 
Decomposition:
 $$ W_k (G_2  ) = L_{A_1} (\frac{1}{2} \Lambda _0) \oplus  L_{A_1} (-\frac{5}{2}\Lambda _0 + 3 \Lambda_1).$$

\item{Case \eqref{F4}:} $\g^\natural$ of Type $C_3$, $\g_{-1/2}=V_{C_3}(\omega_3)$
$$\g_{-1/2}\otimes \g_{-1/2}  = V_{C_3}(0 ) + V_{C_3}(2 \omega_1) + V_{C_3}( 2 \omega_3). $$
Values of $h_\mu$'s:
$$h_{ 2 \omega_1} = \frac{8}{5},   \quad h_{ 2 \omega_3} = \frac{18}{5}.$$

Decomposition:
$$ W_k (F_4  ) = L_{C_3} (-\frac{3}{2} \Lambda _0) \oplus  L_{C_3} (-\frac{5}{2}\Lambda _0+  \Lambda_3).$$

\item{Case \eqref{E6}: } $\g^\natural$ of Type $A_5$, $\g_{-1/2}=V_{A_5}(\omega_3)$
$$\g_{-1/2}\otimes \g_{-1/2}  = V_{A_5}(0 ) + V_{A_5}( \omega_1 + \omega_5) + V_{A_5}(  \omega_2+ \omega_4) + V_{A_5}(  2 \omega_3) . $$
Values of $h_\mu$'s:
$$h_{  \omega_1 + \omega_5} = \frac{12}{7},  \quad h_{  \omega_2+ \omega_4} = \frac{20}{7}, \quad h_{ 2 \omega_3} = \frac{24}{7}.$$

Decomposition:
$$ W_k (E_6 ) = L_{A_5} (-\frac{5}{2}\Lambda _0) \oplus  L_{A_5} (-\frac{7}{2}\Lambda _0 +  \Lambda_3).$$

\item{Case \eqref{E7}:} $\g^\natural$ of Type $D_6$, $\g_{-1/2}=V_{D_6}(\omega_6)$
$$\g_{-1/2}\otimes \g_{-1/2} = V_{D_6}(0 ) + V_{D_6}( \omega_2 ) + V_{D_6}(  \omega_4) + V_{D_5}(  2 \omega_6) . $$
Values of $h_\mu$'s:
$$h_{2 \omega_6} = \frac{36}{11},  \quad h_{\omega_4} = \frac{32}{11}, \quad  h_{\omega_2} = \frac{20}{11}.$$

Decomposition:
$$ W_k (E_7 ) = L_{D_6} (-\frac{9}{2} \Lambda_0) \oplus L_{D_6} (-\frac{11}{2} \Lambda_0 + \Lambda_6).$$

\item{Case \eqref{E8}:}  $\g^\natural$ of Type $E_7$, $\g_{-1/2}=V_{E_7}(\omega_7)$
$$\g_{-1/2}\otimes \g_{-1/2}   = V_{E_7}(0 ) + V_{E_7}( \omega_1 ) + V_{E_7}(  \omega_6) + V_{E_7 }(  2 \omega_7) . $$
Values of $h_\mu$'s:
$$h_{2 \omega_7} =  \frac{60}{19} ,  \quad h_{\omega_6} = \frac{56}{19} , \quad  h_{\omega_1} =  \frac{36}{19}.$$

Decomposition:
$$ W_k (E_8 ) = L_{E_7 } (-\frac{17}{2} \Lambda_0) \oplus L_{E_7} (-\frac{19}{2} \Lambda_0 + \Lambda_7).$$

\item{Case \eqref{F(4) }:} $\g^\natural$ of Type $B_3$, $\g_{-1/2}=V_{B_3}(\omega_3)$
\begin{align*}
\g&_{-1/2}\otimes \g_{-1/2}\\&=V_{B_3}( \omega_3) \otimes V_{B_3}( \omega_3)  = V_{B_3}(2  \omega_3) + V_{B_3}(  \omega_2) + V_{B_3}( \omega_1) +V_{B_3}( 0).
\end{align*}
Values of $h_\mu$'s:
$$  \quad  h_{ 2 \omega_3} = \frac{24}{7} , \quad  h_{ \omega_2} = \frac{20}{7} , \quad  h_{ \omega_1} = \frac{12}{7},$$
which are not integers.
 Decomposition
 $$  W_k (F(4)  ) = L_{B_3} (-\frac{13}{4} \Lambda _0) \oplus  L_{B_3} (-\frac{17}{4}\Lambda _0 +  \Lambda_3).$$
\item{Case \eqref{osp(4,2n)}:} $\g^\natural$ of Type $A_1\times C_n$, $\g_{-1/2}=V_{A_1}(\omega_1)\otimes V_{C_n}(\omega_1)$
\begin{align*}
\g&_{-1/2}\otimes \g_{-1/2}\\&= (V_{A_1}(2  \omega_1) + V_{A_1}( 0))\otimes(V_{C_{n}}(0 ) + V_{C_{n}}( \omega_2 ) + V_{C_{n}}( 2 \omega_1)).
\end{align*}
Values of $h_\mu$'s:
%
$$ h_{2 \omega_1, 0} =  \frac{4}{3} ,  \quad h_{0, \omega_2} = \frac{4n}{2n+1}, \  h_{0, 2\omega_1} = 
\frac{4n+4}{2n+1},$$
$$h_{2 \omega_1, \omega_2} = \frac{4}{3} + \frac{4n}{2n+1} ,  \quad h_{2 \omega_1, 2 \omega_1} =\frac{4}{3} + \frac{4n+4}{2n+1}.$$
which are not integers if $n\ge2$.

Decomposition:
\begin{align*}{W}_{k} (osp(4|2n) ) &= L_{A_1}( -\frac{1}{2} \Lambda_0) \otimes L_{C_n}(- \frac{2n+3}{4} \Lambda_0)\\& \oplus L_{A_1}( -\frac{3}{2}  \Lambda_0 + \Lambda_1) \otimes 
L_{C_n}( -\frac{2n+7}{4} \Lambda_0 + \Lambda_1) .
\end{align*}
\item{Case \eqref{C-emb}: } $\g^\natural$ of Type $C_n$, $\g_{-1/2}=V_{C_n}(\omega_1)$, 
$$\g_{-1/2}\otimes \g_{-1/2} = V_{C_n}(2  \omega_1) + V_{C_n}( \omega_2) + V_{C_n}( 0). $$
Values of $h_\mu$'s:
$$ h_{2 \omega_1} = \frac{ 6 (n+1)}{ 2n +1}, \quad   h_{\omega_2} =  \frac{ 6 n }{ 2n +1}. $$
For $n \ge 2$ we have that $h_{2 \omega_1} $ and  $ h_{\omega_2}$ are non-integral.

Decomposition:
$$ W_k (C_{n+1}  )  = L_{C_n} (  -\frac{4 n +5}{6} \Lambda_0) \oplus   L_{C_n} ( -\frac{4 n +11}{6} \Lambda_0 + \Lambda_1) . $$

\item{ Case \eqref{spo(2,2n)}:} $\g^\natural$ of Type $D_n$, $\g_{-1/2}=V_{D_n}(\omega_1)$,
$$\g_{-1/2}\otimes \g_{-1/2} = V_{D_n} (0) \oplus V_{D_n} (2 \omega_1) \oplus V_{D_n} (\omega_2).$$ 
Values of $h_\mu$'s:
$$ h_{2 \omega_1}=  3 + \frac{3}{2n-1}, \quad  h_{\omega_2} = 3 - \frac{3}{2n-1}.  $$
These values are  non-integral for $n\ge3$.

Decomposition:
$$ W_k (spo(2|2n)) =L_{D_n} (-\frac{4n -5}{3} \Lambda_0) \oplus L_{D_n} (-\frac{4n -2}{3} \Lambda_0 + \Lambda_1).$$

\item{Case \eqref{spo(2,2n+1)}:} $\g^\natural$ of Type $B_n$, $\g_{-1/2}=\begin{cases}V_{B_n}(\omega_1)&\text{if $n\ge2$}\\
V_{A_1}(2\omega_1)&\text{if $n=1$}
\end{cases}$, 
$$\g_{-1/2}\otimes \g_{-1/2} =\begin{cases} V_{B_n} (0) \oplus V_{B_n} (2 \omega_1) \oplus V_{B_n} (\omega_2)&\text{if $n\ge2$}\\V_{A_1} (0) \oplus V_{A_1} (2 \omega_1) \oplus V_{A_1} (4 \omega_1)&\text{if $n=1$}
\end{cases}.$$ 
Values of $h_\mu$'s:
$$ h_{2 \omega_1}=  3 + \frac{3}{2n}, \quad  h_{\omega_2} = 3 - \frac{3}{2n}.  $$
These values are  non-integral for $n\ge1$.

Decomposition:
$$ W_k (spo(2|2n+1)) =L_{B_n} ( -\frac{4n -3}{3} \Lambda_0) \oplus L_{B_n} ( -\frac{4n}{3} \Lambda_0 + \Lambda_1),$$ 
and for $n =1$
$$W_k (spo(2|3)) =L_{A_1} (-\frac{2}{3} \Lambda_0) \oplus L_{A_1} (-\frac{8}{3} \Lambda_0 + 2 \Lambda_1)$$ 
%

\item{Case \eqref{spo1}} $\g^\natural=spo(n-2|m)$, $\g_{-1/2}=\C^{n-2|m}$. We have 
\begin{equation}\label{s+l}\g_{-1/2}\otimes \g_{-1/2}= S^2\C^{n-2|m} \oplus \wedge^2 \C^{n-2|m}.\end{equation}
As $\g^\natural$-module, the first summand in the r.h.s. of \eqref{s+l} is the adjoint representation of $\g^\natural$ (which is irreducible, since $\g^\natural$ is simple), and the second summand in the sum of a trivial representation and an irreducible summand.  Fix in $\g$ the distinguished set of positive roots $\Pi_B$ if $m$ is odd and $\Pi_{D1}$ if $m$ is even (notation as in \cite[4.4]{GKMP}).
This choice induces  on $\g^\natural$ a distinguished set of positive roots and, with respect to it, the nonzero highest weights of the irreducible $\g^\natural$-modules appearing in \eqref{s+l} are $2\d_2$, $\d_2+\d_3$.
Values of $h_\mu$'s:
$$ h_{2 \d_2}=  3( 1+ \frac{1}{n-m-1}), \quad  h_{\d_2+\d_3} = 3( 1- \frac{1}{n-m-1}).  $$
This values are  integers if and only if $m=n+2,n,n-2,n-4$.

Case \eqref{spo3}:
$\g^\natural=osp(m-4|n)\oplus sl(2)$, $\g_{-1/2}=\C^{m-4|n}\otimes C^2$. We have 
\begin{equation}\g_{-1/2}\otimes \g_{-1/2}= (\wedge^2 \C^{m-4|n}\oplus S^2\C^{m-4|n} )\otimes (sl(2)\oplus \C).\end{equation}
As $osp(m-4|n)$-modules, $\wedge^2 \C^{m-4|n}$ is the adjoint representation (which is irreducible, since $osp(m-4|n)$ is simple), and $S^2\C^{m-4|n}$ is  the sum of a trivial representation and an irreducible summand. If $m=2t+1$ is odd, we fix in $\g$ the set of positive roots corresponding to the diagram \cite[(4.20)]{GKMP} with $\a_t$ odd isotropic, the short root odd non-isotropic, and the other roots even. If $m$ is even we choose the set of positive roots corpsonding to the diagram $\Pi_{D2}$ of \cite{GKMP}. 
With respect to the induced  set of positive roots for $osp(m-4|n)$ the  highest weight of $\wedge^2 \C^{m-4|n}$ is $\e_3+\e_4$ and  the  highest weight of the nontrivial irreducible component of $S^2\C^{m-4|n} $ is $2\e_3$. The highest weight of $sl(2)$ is $\e_1-\e_2$.
Values of $h_\mu$'s:
$$ h_{2 \e_3,\e_1-\e_2}=\frac{10}{3}+ \frac{2}{m-n-5}, \quad  h_{\e_3+\e_4,\e_1-\e_2} = \frac{10}{3}- \frac{2}{m-n-5},  $$
$$ h_{2 \e_3,0}=2(1+ \frac{1}{m-n-5}), \quad  h_{\e_3+\e_4,0}=2(1- \frac{1}{m-n-5}),\quad   h_{0,\e_1-\e_2}= \frac{4}{3}.$$
This values are not in $\ganz_+$  for  $m,n$ in the range showed in the statement.
 \end{proof}


We now list  the cases that are not covered by Corollary \ref{conditionsfinite1}. We list here only the cases where there is a non-collapsing conformal  level as described in Proposition  
\ref{confnoncollaps}.
 \begin{enumerate}
\item \label{G(3)5/4}$\g$ of type  $G(3)$, $k=\frac{5}{4}$;
\item \label{D(2,1;a)}$\g=D(2,1;a)$ ($a\not\in\{\half,-\half,-\frac{3}{2}\}$), $k=\half$;
\item\label{psl(2|2)1/2}$\g=psl(2|2), k=\half$;
\item\label{spo(n|m)} $\g=spo(m+r|m)$, ($m\ge4$, $r\in\{0,2\}$), $k=-\frac{r+2}{3}$;
\item\label{osp(m|n)1}$\g=osp(n+r|n)$, ($r\in\{-1,3,4,6,8,11\}$), $k=\frac{3-r}{2}$;
\item\label{osp(m|n)2}$\g=osp(m|n)$, $k=\frac{2}{3}(n-m+2)$;
\item $\g=F(4)$, $k=-1$;
\item $\g=G(3)$, $k=-\frac{1}{2}$.
\end{enumerate}
Sometimes ${W}_{k}(\g, \theta)$ still decomposes  finitely as a $\mathcal V_k(\g^\natural)$-module.
 More explicitly, we have  the following result:
 
 \begin{theorem}\label{exeptionalfinite}$\mathcal V_{k} ({\mathfrak g}^{\natural})$ is a simple affine vertex algebra and
  ${W}_{k} (\g,\theta) ^ {\bar i }$, $i=0,1$, are irreducible $\mathcal V_{k} ({\mathfrak g}^{\natural})$-modules in the following cases:
  \begin{enumerate}
  \item\label{exi} $\g=so(8)$, $k=-\frac{5}{2}$;
  \item\label{exii} $\g=D(2,1;1)=osp(4|2)$, $k=\half$;
  \item\label{exiii} $\g=D(2,1;1/4)$, $k=\half$.
  \end{enumerate}
 \end{theorem}
 The proof of Theorem \ref{exeptionalfinite} requires some representation theory of the vertex algebra $V_{-1/2} (sl(2))$.

\begin{rem}   Let  $k=\half$. In the case when $\frac{k-a}{a} $, where $a\in\mathbb Q$,  is an admissible level for $\widehat{sl_2}$ we also expect that ${W}_{k} (\g,\theta)$ is a finite sum of $\mathcal V_k (\g ^{\natural})$-modules, but the decomposition is more complicated.  We think that the methods developed in  \cite{FS} can be applied for this conformal embedding. Here we shall only consider the cases $a = 1$ and $a = 1/4$, where we can apply fusion rules for affine vertex algebras.
\end{rem}
We will  prove  cases (1), (2), (3)  of Theorem \ref{exeptionalfinite}  in Sections \ref{d4}, \ref{osp(4,2)}, and \ref{D(2,1,1/4)}, respectively.

 The next result shows that in case \eqref{psl(2|2)1/2} above we have an infinite decomposition:
 \begin{theorem}
 If $\g=psl(2|2)$ and $k=\half$ then $\mathcal V_k(\g^\natural)$ is simple and $W_{k}(\g)$ decomposes into an infinite direct sum of irreducible  $\mathcal V_k(\g^\natural)$-modules.
 \end{theorem}
\begin{proof}In this case
 ${W}_{k}(\g, \theta)$ is isomorphic to the $N=4$ superconformal vertex algebra $V^{N=4} _c$  with $c=-9$ (cf. \cite{KW}). The explicit realization of this vertex algebra from \cite{A-2014} gives the result.
\end{proof}

\begin{rem} \label{examples-infinite} The remaining open cases are
 
 \begin{itemize}
\item $\g$ of type  $G(3)$, $k=\frac{5}{4}$;
\item $\g=D(2,1;a)$, ($a\not\in\{\half,-\half,-\frac{3}{2},1,\frac{1}{4}\}$), $k=\half$;
\item $\g=so(11)$, $k=-4$;
\item $\g$ of type $B_n$ ($n\ge3$), $k=-\frac{4n-2}{3}$;
\item $\g$ of type $D_n$ ($n\ge5$), $k=-\frac{4n-4}{3}$;
\item $\g=osp(4|n)$ ($n>2$, $n\ne8$), $k=-\frac{4-2n}{3}$.
\end{itemize}





\end{rem}
\subsection{Proof of Theorem \ref{exeptionalfinite}}\label{proofthmexefinite}
As in Theorem  \ref{seminocenter},  we set 
$$\mathcal{U} = \mbox{span} \{G ^{ \{u \} } \ \vert \ u \in \g_{-1/2}\ \},\quad A  = \mathcal V_{k} ({\g}^{\natural}) \cdot \mathcal{U}.$$
 In order to apply Theorem \ref{seminocenter}, we have to prove that 
 $$A\cdot A\subset \mathcal V_k(\g^\natural).$$
We first prove that  $\mathcal V_k(\g^\natural)$ is simple and that $A$ is a simple $\mathcal V_k(\g^\natural)$-module. Since $\mathcal V_k(\g^\natural)$ is admissible, we have that $A\cdot A$ is completely reducible \cite{KW-1988}. Let  
$A\cdot A = \sum_i M_i$ be its  decomposition into simple modules. We have to show  that the only summands appearing are vacuum modules.
This is guaranteed by the fusion rules \eqref{fr-12} presented  in   Subsection \ref{sub-12}. Next we provide details in each of the three cases.

 \subsubsection{ Proof of Theorem  \ref{exeptionalfinite}, case \eqref{exi}. }
\label{d4}

We  claim that, as $\widehat{sl(2)}$-modules, 
$${W}_{-5/2} (so(8),\theta) = L_{sl(2)}(-\tfrac{1}{2} \Lambda_0) ^{ \otimes  3}  \oplus L_{sl(2)}(-\tfrac{3}{2} \Lambda_0 + \Lambda_1)^{\otimes 3}  . $$
Recall that in this case $\g^\natural\simeq sl(2)\oplus sl(2) \oplus sl(2)$.  Then $A$ is a highest weight $ V^{-1/2} (sl(2))^{\otimes 3} $-module with highest weight vector $v_{\lambda_1} \otimes v_{\lambda_1} \otimes v_{\lambda_1}$.


Assume now that $A$ is not irreducible. Then, as observed in Subsection \ref{sub-12} (2), a quotient of $N_{sl(2)} (\lambda_1)$ is either irreducible or it is $N_{sl(2)} (\lambda_1)$ itself. It follows that 
$$A \cong N_{sl(2)} (\lambda_1) \otimes \widetilde{L}_{sl(2)} (\lambda_1) \otimes \widetilde { \widetilde{L} } _{sl(2)}(\lambda_1) $$
where $\widetilde{L}_{sl(2)} (\lambda_1) $ and  $\widetilde {\widetilde{L}} _{sl(2)} (\lambda_1) $ are certain quotients of $N_{sl(2)} (\lambda_1)$ . 

Let  $$ w_3 = v_{\lambda_3} \otimes v_{\lambda_1} \otimes v_{\lambda_1}, \qquad W_3 = U(\g ^{\natural} )  w_3 \subset  A. $$

By using tensor product decompositions
\begin{align*} V_{sl(2)} (3 \omega_1) \otimes V_{sl(2)} (\omega_1) &= V_{sl(2)} (4 \omega_1) \oplus V_{sl(2)} (2 \omega_1), \\
V_{sl(2)} ( \omega_1) \otimes V_{sl(2)} (\omega_1) &= V_{sl(2)} (2 \omega_1) \oplus V_{sl(2)} (0),\end{align*}
and arguing as in the proof of Theorem \ref{seminocenter} we see that 
$ {\mathcal U}\cdot W_3$ cannot contain  primitive vectors of conformal weight $\le 3$. Since the conformal weight of all elements of $W_3$ equals the conformal weight of $w_3$ which is $\frac{7} {2}$ and the conformal weight of all elements of $\mathcal U$ is $\frac{3}{2}$, we conclude that
$$ \mathcal{U} _{(n)} W_3 = 0 \qquad (n \ge 1). $$
This implies that $w_3$ is a non-trivial singular vector in ${W}_{-5/2} (so(8),\theta)$. A contradiction.
Therefore $W_3 = 0 $ and 
$A \cong L (\lambda_1) ^ {\otimes 3 }.$



We will now  show that $\mathcal V_{-5/2} (\g ^{\natural})$ is simple. If not, since a quotient of  $V^{-1/2} (sl(2))$ is either simple or $V^{-1/2} (sl(2))$ itself, we have 
$$ \mathcal V_{-5/2} (\g ^{\natural}) =  V^{-1/2} (sl(2)) \otimes  {\widetilde V}_{-1/2} (sl(2)) \otimes  \widetilde {\widetilde V}_{-1/2} (sl(2)),$$
where  ${\widetilde V}_{-1/2} (sl(2))$  and $\widetilde {\widetilde V}_{-1/2} (sl(2))$ are certain quotients of $V^{-1/2} (sl(2))$. 

Set  $$ w_4 = v_{\lambda_4} \otimes {\bf 1}  \otimes {\bf 1} , \qquad W_4 = U(\g ^{\natural} )w_4 \subset  \mathcal V_{-5/2} (\g ^{\natural}). $$
By using fusion rules again we see that  ${\mathcal U} _{(1)} W_4 = W_3 =0$.
So $w_4 $ is a singular vector in ${W}_{-5/2} (so(8),\theta)$, a contradiction.

Therefore $\mathcal V_{-5/2} (\g ^{\natural})= V_{-1/2} (sl(2)) ^{\otimes 3}$.   By using fusion rules  (\ref{fr-12})
 we easily get that 
$$ \mathcal V_{-5/2} (\g ^{\natural}) \oplus A= V_{-1/2} (sl(2)) ^{\otimes 3} \oplus L (\lambda_1) ^ {\otimes 3 }$$
is a vertex subalgebra of ${W}_{-5/2} (so(8),\theta)$. Since this subalgebra contains all generators of ${W}_{-5/2} (so(8),\theta)$, the claim 
follows. \qed

\subsubsection{ Proof of Theorem  \ref{exeptionalfinite}, case \eqref{exii}. } \label{osp(4,2)} 
The proof is similar to case \eqref{exi}. 

We claim that,  as $\widehat{sl(2)}$-modules,
\begin{align*}&{W}_{1/2} (osp(4 \vert 2 ),\theta)\\&=L_{sl(2)}(-\tfrac{1}{2} \Lambda_0) \otimes  L_{sl(2)}(-\tfrac{5}{4} \Lambda_0)\oplus L_{sl(2)}(-\tfrac{3}{2} \Lambda_0 + \Lambda_1) \otimes  L_{sl(2)}(-\tfrac{9}{4} \Lambda_0+ \Lambda_1) . \end{align*}
Let $\mathcal{U}, A$ be as in Subsection \ref{d4}. Then $A$ is a highest weight $\mathcal V_{1/2} ({\g}^{\natural}) $-module with highest weight vector
$v_{\lambda_1} \otimes v_{\tilde \lambda_1}$ where $ \tilde \lambda _1 = -\frac{9}{4} \Lambda_ 0 + \Lambda_1$.
Assume now that $ V ^{-1/2} (sl(2)) v_{\lambda_1}$ is not simple. Then
$$A \cong N_{sl(2)} (\lambda_1) \otimes \widetilde{L} _{sl(2)} ( \tilde \lambda_1),   $$
where $\widetilde{L}_{sl(2)} (\tilde \lambda_1) $ is certain a  quotient of $N_{sl(2)} (\tilde \lambda_1)$.

Set  $$ w_3 = v_{\lambda_3} \otimes v_{\tilde \lambda_1},   \qquad W_3 = U(\g ^{\natural} )  w_3 \subset A . $$
The same argument of case (1), using fusion rules and evaluation of conformal weights, shows that  $w_3$ is a non-trivial singular vector in 
${W}_{1/2}(osp(4 \vert 2 ),$ $\theta)$, a contradiction.
Therefore $W_3 = 0 $ and 
$$A \cong L (\lambda_1)  \otimes  \tilde L (\tilde \lambda_1) .$$

Assume next that $V^{-1/2} (sl(2)) \cdot {\bf 1}$ is not simple. Then 
$$ \mathcal V_{1/2} (\g ^{\natural}) =  V^{-1/2} (sl(2)) \otimes  {\widetilde V}^{-5/4} (sl(2)) $$
where  ${\widetilde V}^{-5/4} (sl(2))$ is a certain quotient of $V^{-5/4} (sl(2))$. 

Set  $$ w_4 = v_{\lambda_4} \otimes {\bf 1}   , \qquad W_4 = U(\g ^{\natural} ) w_4 \subset  \mathcal V_{1/2} (\g ^{\natural}). $$
By using fusion rules again we see that  ${\mathcal U} _{(1)}W_4 = W_3 =0$.
So $w_4 $ is a singular vector in ${W}_{1/2}(osp(4 \vert 2),\theta)$. A contradiction.

Therefore $\mathcal V_{1/2} (\g ^{\natural})= V_{-1/2} (sl(2)) \otimes \widetilde V ^{-5/4} (sl(2))$.  By using fusion rules 
 (\ref{fr-12}) 
 we easily get that 
$$ \mathcal V_{1/2} (\g ^{\natural}) \oplus A = {W}_{1/2} (osp(4 \vert 2 ),\theta).$$
In particular $ V_{1/2} (\g ^{\natural}) $ is a simple vertex algebra and  $A$ is its simple module. The claim follows. \qed

\subsubsection{ Proof of Theorem  \ref{exeptionalfinite}, case \eqref{exiii}.  }\label{D(2,1,1/4)}

In case \eqref{exiii} we have 
\begin{align*}&{W}_{1/2} (D(2,1;1/4),\theta)
= \\&L_{sl(2)}( \Lambda_0) \otimes  L_{sl(2)}(-\tfrac{7}{5} \Lambda_0)\oplus L_{sl(2)}(  \Lambda_1) \otimes  L_{sl(2)}(-\tfrac{12}{5} \Lambda_0+ \Lambda_1),\end{align*}
and the proof is completely analogous to that  of case \eqref{exii}.

\section{ The vertex algebra $\mathcal R^{(3)}$ and proof of Theorem \ref{A3special} }

\label{sect-r3}

 In this section we will  present an explicit realization of the vertex algebra   $W _k(sl(4),\theta)$ and prove that it is isomorphic to the vertex algebra $\mathcal R ^{(3)}$ from \cite{A-2014}. In this way we prove Conjecture 2 from \cite{A-2014}. Then we apply this new realization to construct explicitly infinitely many singular vectors in each charge component $ {W}_{k} ^{(i)},$  proving Theorem \ref{A3special}.

 \subsection{Definition of $\mathcal{R}^{(3)}$ } Let us first recall  the definition of  the vertex algebra $\mathcal{R}^{(3)}$  introduced in Section 12  of \cite{A-2014}.  
Let $V_L = M(1) \otimes {\C}[L] $ be the generalized lattice vertex algebra  (cf. \cite{BK}, \cite{DL})  associated to the (non-integral) lattice
$$ L= \Z \alpha + \Z \beta + \Z \delta + \Z \varphi,$$
with non-zero inner products
$$ \langle  \alpha, \alpha \rangle  = - \langle  \beta, \beta \rangle  = 1, \quad  \langle  \delta , \delta  \rangle  = - \langle  \varphi, \varphi \rangle  = \frac{2}{3}.$$

Set $\alpha_1 = \alpha + \beta $, $\alpha_2 = \frac{3}{2} (\delta + \varphi)$,  $\alpha_3 = \frac{3}{2} (\delta - \varphi)$, and
$$ D = \Z \alpha_1 + \Z \alpha_2 + \Z \alpha_3.$$
Then $D$ is an even integral lattice. We choose a bi-multiplicative 2-cocycle $\varepsilon$ such that for every $\gamma_1, \gamma_2 \in D$ we have
$$ \varepsilon (\gamma_1, \gamma_2)  \varepsilon (\gamma_2, \gamma_1) = (-1) ^{ \langle \gamma_1, \gamma_ 2 \rangle  }. $$
We fix the following choice of  the cocycle:
\bea && \varepsilon (\alpha_1, \alpha_i) =  \varepsilon (\alpha_i, \alpha_1)  =\varepsilon (\alpha_i, \alpha_i ) =1 \quad (i=1,2,3) \nonumber \\
 && \varepsilon (\alpha_2, \alpha_3) = - \varepsilon (\alpha_3, \alpha_2) = 1. \nonumber \eea  
 This cocycle can be extended to a  2-cocyle on  $L$ by bimultiplicativity.
Then we have
\bea  && \varepsilon (\alpha + \beta - 3 \delta , \alpha_3)  = \varepsilon (\alpha_1 -\alpha_2 -\alpha_3, \alpha_3)  = 1, \nonumber \\
 &&  \varepsilon (\alpha + \beta - 3 \delta , \alpha_2)  = \varepsilon (\alpha_1 -\alpha_2 -\alpha_3, \alpha_2)  = - 1. \nonumber 
 \eea

Let $C_{\varepsilon} [D] $ be  the twisted group algebra associated to the lattice $D$ and cocycle $\varepsilon$.
  We consider the lattice type vertex algebra  $$V_D  ^{ext} =  M (1) \otimes C_{\varepsilon} [D], $$ which is realized as a  vertex subalgebra of $V_L$. (Note that $V_D  ^{ext} $ contains the complete  Heisenberg vertex subalgebra  $M(1)$  of $V_L$, and that the lattice $D$ has three generators.) All calculations below will be done in this vertex algebra.

For $\gamma\in D$   we define  the following elements of the Heisenberg vertex algebra $M(1)$:
$$ S_2 (\gamma) = \frac{1}{2} ((\gamma_{(-1)} )^2 + \gamma_{(-2)} ), \ \ S_3 (\gamma) = \frac{1}{6} (\gamma _{(-1)} ^ 3 + 3 \gamma_{(-1)} \gamma_{(-2)} + 2  \gamma_{(-3)} ). $$
 First we recall that 
 the vertex subalgebra $\mathcal V$ of $V_D ^{ext} $, generated by 
\bea
 e  &=&  e^{\alpha + \beta}, \\  
 h & = &  -2 \beta + \delta, \label{def-h} \nonumber \\
 f &=&  ( -\frac{2}{3} (\alpha _{(-1)} ^2 + \alpha_{(-2)} )  - \alpha_{(-1)} \delta _{(-1)}  + \frac{1}{3}\alpha_{(-1)} \beta _{(-1)} ) e ^{-\alpha -\beta}, \nonumber  \\
 j  &=&  \varphi, \nonumber 
 \eea
 is an affine vertex algebra. More precisely, it is isomorphic to $ M_{\varphi}(-2/3)\otimes V_{-5/3}(sl(2))$
 (Note that $k = -5/3$ is a generic level, i.e. $V_{-5/3}(sl(2)) = V^{-5/3} (sl(2))$, cf. \cite{GK}).
 
 Let $Q= e^{\alpha + \beta - 3 \delta} _{(0)}$ be the screening operator (cf. \cite{A-2014}). Note that $Q$ is a derivation of the vertex algebra $V_D^{ext}$. We also have that the Sugawara Virasoro vector $\omega^{\mathcal V}_{sug}$ of $\mathcal V$ maps to 
 $$ \left(\frac{1}{2} (\alpha_{(-1)} ^2 - \alpha_{(-2)}  - \beta_{(-1)} ^2 + \beta_{(-2)} ) + \frac{3}{4} (\delta_{(-1)} ^2 - 2 \delta_{(-2)} - \varphi_{(-1)}^2 )\right)\vac .$$
We define  $\mathcal{R}^{(3)}$ to be the vertex subalgebra of $V_D^{ext}$  generated by the generators of $\mathcal V$ and  the following four even vectors of conformal weight $3/2$:
\bea
 E^1 & = & e^{\tfrac{3}{2} (\delta + \varphi)}, \nonumber \\
  E^2 & = & Q e^{\tfrac{3}{2} (\delta - \varphi)}   = S_2 (\alpha + \beta -3 \delta) e ^{-\tfrac{3}{2} (\delta + \varphi) + \alpha + \beta }  , \nonumber \\
  F^1 & = & f_{(0)} E^1 =  -\alpha_{(-1)} e^{ -\alpha - \beta  + \tfrac{3}{2} (\delta + \varphi)  },\nonumber  \\
  F ^2 & = & f_{(0)} E^2  =(  - \alpha_{(-1)} S_2 (\alpha + \beta -3 \delta) + S_3 (\alpha+\beta - 3 \delta ) )e^{-\tfrac{3}{2} (\delta+ \varphi)}.\nonumber 
\eea
The vertex algebra  $ \mathcal R ^{(3)} $ satisfies the following properties:
 \begin{itemize}
 \item $ \mathcal R ^{(3)} $ is integrable, as a module over $sl(2)$. \label{cond-int}
 \item  $ \mathcal R ^{(3)} $ has finite-dimensional weight spaces with respect to $(\omega^{\mathcal V}_{sug} )_0$. The conformal weights
lie in  $\frac{1}{2} \Z_{\ge 0}$.\label{cond-l0}
 \item $ \mathcal R ^{(3)}$ is contained in the following subalgebra of $V_D^{ext}$:  \label{cond-half-lattice} $$ M \otimes \Pi(0), $$
 where $M$ is the Weyl vertex algebra (i.e., the algebra of symplectic bosons \cite{KacV}) generated by
 $$a =e ^{\alpha + \beta}, a^* = - \alpha_{(-1)} e^{-\alpha -\beta}, $$
 and $ \Pi (0) $ is the "half-lattice" vertex algebra  $$ \Pi(0) =    M_{\delta,\varphi} (1) \otimes  {\C}[\Z \ \tfrac{3(\delta + \varphi)}{2}] $$
containing the Heisenberg vertex algebra $ M_{\delta,\varphi} (1)$  generated by $\delta$ and $\varphi$ (cf. \cite{A-2014}).
  \end{itemize}
 
 Let $(M \otimes \Pi(0) ) ^{int}$  denote the maximal $sl(2)$-integrable submodule of $M \otimes \Pi(0) $. It is clear that it is a vertex subalgebra of $M \otimes \Pi(0)$.

We shall prove the following result.

\begin{theorem}\label{proof-conjecture}

\item[(1)] There is a vertex algebra homomorphism $ {W}^ {k} (sl(4),\theta) \rightarrow \mathcal{R}^{(3)}$.  
 
 \item[(2)] $ \mathcal{R}^{(3)}$ is a simple vertex algebra, i.e,   $ {W}_ {k} (sl(4),\theta) = \mathcal{R}^{(3)}$. 
 
 \item[(3)] $ \mathcal{R}^{(3)} \cong (M \otimes \Pi(0) ) ^{int}$.
 
 \end{theorem}

\begin{rem} Theorem \ref{proof-conjecture} gives a positive answer to Conjecture 2 from \cite{A-2014}. The representation theory of $ \mathcal R^{(p)}$ for $p >3$  and its relation with $C_2$-cofinite vertex algebras appearing in LCFT (such as triplet vertex algebras) will be studied in  \cite{A-2016}.
\end{rem}

\subsection { $\lambda$-brackets for $ \mathcal{R}^{(3)}$ }

\begin{proposition}  \label{bracket-for-r3} We have the following $\lambda$-brackets:
\bea  [\,{E^i}_{\lambda}E^i ]& =& [{F^i} _{\lambda} F^i ]  = 0\quad (i=1,2), \nonumber  \\{}
[{E^1}_{\lambda} E^2] &=& 3 ( \partial  e + 3: j e:) + 6 \lambda e,\nonumber  \\{}
[{F^1}_{\lambda} F^2 ]&= &-3 ( \partial  f  + 3: j f: ) - 6 \lambda f  ,\nonumber \\{} 
[{E ^ 1}_{\lambda} F^1] & = & 0, \nonumber \\{} 
[{E ^ 1}_{\lambda} F^2 ]& = &  -3 ( \omega^{\mathcal V}_{sug}   + \frac{1}{2}  ( \partial h  + 3: j h:   - 6: j j:  -5  \partial j )  ) )\nonumber\\&& + 3 \lambda (   - h + 5j  ) + 5 \lambda^2 ,\nonumber  \\{} 
[{E ^ 2}_{\lambda} F^1 ]& = & -3 ( \omega^{\mathcal V}_{sug}   + \frac{1}{2}  ( \partial h  -  3: j  h:   - 6: jj: + 5  \partial j )  ) ) \nonumber\\&&-  3 \lambda (    h + 5 j  ) +  5 \lambda^2 ,  \nonumber  \\{} 
[{E ^ 2}_{\lambda} F^2] & = & 0 .\nonumber
\eea
\end{proposition}
\begin{proof}
The proof uses the standard computations in   lattice vertex algebras \cite{KacV}. Let us discuss the calculation of $[{E ^ 1} _{\lambda} F^2 ]$ and of $[{E ^ 2} _{\lambda} F^1 ]$.

For $[{E ^ 1}_{\lambda} F^2 ]$,  the only difficult part  is to compute    ${E^1} _{(0)}  F^2$. We have
\bea 
{E^1}_{(2)} F^2  &=& 10, \nonumber \\
 {E^1}_{(1)} F^2  &=& - h + 5 \varphi  = - h + 5 j, \nonumber \\
{E^1}_{(0)} F^2  
&=& - 9 \alpha_{(-1)} (\delta + \varphi ) -  3 \alpha_{(-1)} (\alpha+ \beta- 3 \delta ),  \nonumber  \\
&& +10 S_2 (\tfrac{3}{2} (\delta + \varphi) ) \vac+  9 (\alpha_{(-1)} + \beta_{(-1)} - 3 \delta_{(-1)} ) (\delta + \varphi)\nonumber \\
&& + 3S_2 (\alpha + \beta - 3 \delta)\vac  \nonumber \\
& = & -3  ( \omega_{sug} + 1/2 ( h_{(-2)}  + 3 \varphi_{(-1)} h_{(-1)}  - 6 \varphi_{(-1)} ^2  - 5  \varphi_{(-2)}  ) \vac).  \nonumber 
\eea

For  the calculation of $[{E ^ 2}_{\lambda} F^1 ]$, we shall use the fact that $Q$ is a derivation in the  lattice vertex algebra $V_D$.
Set   $$\overline  E^1  = e^{\tfrac{3}{2} (\delta - \varphi)},  $$
$$  \overline  F^2   =  Q F^1 =  - ( - \alpha_{(-1)} S_2 (\alpha + \beta -3 \delta) + S_3 (\alpha+\beta - 3 \delta ) )e^{-\tfrac{3}{2} (\delta-  \varphi)}.$$
Note that the minus sign in front of the r.h.s. of the formula above comes from the cocycle computation
$$ \varepsilon (\alpha + \beta - 3 \delta, - \alpha - \beta + \frac{3}{2} (\delta + \varphi)  )= \varepsilon (\alpha_1 - \alpha_2 -\alpha_3, -\alpha_1 + \alpha_2 )  = -1. $$ 
Next, we have
$$ [{E^2}_{\lambda} F^ 1] =  Q  [{e^{\tfrac{3}{2} (\delta - \varphi)}}  _{\lambda} F^1] -  [{e^{\tfrac{3}{2} (\delta - \varphi)} } _{\lambda} Q F^1] = - [{ \overline  E^1 }   _{\lambda} \overline F^2]. $$
 The calculation of  $[{\overline  E^1 }_{\lambda} \overline F^2]$ is essentially the same as for  $[{E ^ 1}_{\lambda} F^2 ]$ (we just replace $j$ by  $-j$). Now we have 
\bea
&& [{E^2}_{\lambda} F^ 1 ]= -  [{\overline  E^1}_{\lambda} \overline F^2 ]\nonumber \\
& = & ( -3 ( \omega_{sug}   + \frac{1}{2}  ( \partial h  - 3: j   h:   - 6 :jj: + 5  \partial  j)  ) ) + 3 \lambda (   - h -  5 j   ) + 5 \lambda^2  ) \nonumber \\
&=&   -3 ( \omega_{sug}   + \frac{1}{2}  ( \partial h  - 3 :j   h:   - 6 :jj:+ 5  \partial  j)  ) ) + 3 \lambda (   - h -  5 j   ) + 5 \lambda^2  \nonumber \eea
The claim follows.
\end{proof}
 
\subsection {The homomorphism  $ \Phi: W^k (sl(4),\theta) \rightarrow \mathcal{R}^{(3)}$}
Recall from Example \ref{ex-sl4} that the vertex algebra $W^k (sl(4),\theta)$ is generated  by  the Virasoro vector $\omega$ of central charge $c(k)=15 k / (k+4) - 6k,$  four even generators $ J ^{\{e_{2,3} \} } ,  J ^{\{e_{3,2} \} } ,  J ^{\{ e_{2,2} - e_{3,3} \} },$ $ J^ {\{ c\} } $ of conformal weight  $1$, 
 and four  even vectors  $   G ^{\{e_{2,1} \}  } ,    G ^{\{e_{3,1} \}  } ,   G ^{\{e_{4,2} \}  }, $ $   G ^{\{e_{4,3} \}  }  $ of conformal weight $3/2$.


By comparing $\lambda$-brackets from Proposition  \ref{bracket-for-r3}  and $\lambda$-brackets for the vertex algebra 
$W^{-8/3} (sl(4),\theta)$ we get the following result:

\begin{proposition} \label{realization-r3} Let $k =-8/3$.
There is a vertex algebra homomorphism
\bea
\Phi :  W^k (sl(4),\theta) & \rightarrow &    \mathcal{R}^{(3)} \nonumber
\eea
such that
$$
 J ^{\{e_{2,3} \} }  \mapsto e,\quad 
 J ^{\{e_{3,2} \} } \mapsto  f,\quad  J ^{\{ e_{2,2} - e_{3,3} \} } \mapsto  h ,\quad
 J^ {\{ c\} }  \mapsto j ,$$
$$
 G ^{\{e_{2,1} \}  }  \mapsto  \frac{ \sqrt{2}}{3}   E^1,\   G ^{\{e_{3,1} \}  }  \mapsto  \frac{ \sqrt{2}}{3}   F^1,\ G ^{\{e_{4,3} \}  }  \mapsto  \frac{ \sqrt{2}}{3} E^2, \ 
  G ^{\{e_{4,2} \}  } \mapsto   -\frac{ \sqrt{2}}{3}  F^2, 
  $$
  $$
  \omega \mapsto  \omega^{\mathcal V}_{sug}.
$$
\end{proposition}
\begin{proof}
It is enough to check $\lambda$-brackets from Example  \ref{ex-sl4}  in the case $k=-8/3$. In particular,   taking into account that 
\begin{align*}
\omega& = \omega_{sug} = \frac{3}{2} ( 2 : J ^{\{e_{2,3} \} }  J ^{\{e_{3,2} \} } :  -  \partial J ^{\{ e_{2,2} - e_{3,3} \} }\\& + \frac{1}{2} : J ^{\{ e_{2,2} - e_{3,3} \} }J ^{\{ e_{2,2} - e_{3,3} \} } : ) -\frac{3}{4}:J^ {\{ c\} }J^ {\{ c\} }: ,
\end{align*}
we get
\bea  &&[{G^{\{e_{2,1}\}}}\!_\lambda G^{\{e_{4,2}\}}] = \frac{2}{3} \omega-2   :J^ {\{ c\} }J^ {\{ c\} }: +  \frac{1}{3}   \partial J ^{\{ e_{2,2} - e_{3,3} \} }  -\frac{5}{3} \partial J^ {\{ c\} }  \nonumber \\
& & +:J^{\{c\}}J^{\{e_{2,2} - e_{3,3}\}}:+ \frac{2}{3}  \lambda  (   -5  J^ {\{ c\} } +  J ^{\{ e_{2,2} - e_{3,3} \} } ) - \frac{10}{9} \lambda ^2  \nonumber \\ 
& & =  - \frac{2}{9}  ( -3 ( \omega^{\mathcal V}_{sug}   + \frac{1}{2}  ( \partial  h  + 3: j h:   - 6 j   ^2 -5  \partial j )  ) ) + 3 \lambda (   - h + 5j  ) + 5 \lambda^2). \nonumber \\
& &= -  \frac{2}{9}[{E^1}_{\lambda} F^ 2]. \nonumber
  \eea 
All other $\lambda$-brackets  are checked similarly.
  \end{proof}

Proposition \ref{realization-r3}  implies that $\mathcal R^{(3)}$ is conformally embedded into a certain quotient of  $W^k (sl(4),\theta)$. In the following subsection we will prove that $\mathcal R^{(3)}$ is isomorphic to $W_k (sl(4),\theta)$.

\subsection{ Simplicity of $\mathcal{R}^{(3)}$  and proof of Theorem \ref{proof-conjecture} } 

Our proof of simplicity is similar to the proof of simplicity of the $N=4$ superconformal vertex algebra realized in \cite{A-2014}. As a tool we shall use the theory of Zhu  algebras associated to the Neveu-Schwarz sector of $\frac{1}{2}{\Z_{\ge 0}}$-graded vertex algebras. Let $V$ is a $\frac{1}{2}{\Z_{\ge 0}}$-graded vertex algebra
and $A(V) = V/O(V)$ the associated Zhu  algebra. Let $[a] = a+ O(V)$  (cf. Subsection  \ref{sect-zhu}).
 
 \begin{lemma} \label{Zhus-relation}\ 
 \begin{enumerate}
\item The Zhu  algebra $A(\mathcal R ^{(3)})$  is isomorphic to a  quotient of $U(gl(2))$.
 \item In the Zhu  algebra $A(\mathcal R ^{(3)})$ the following relation holds:
 $$ [e] ([\omega] + \frac{2}{3} - \frac{3}{2} [j] ^2 ) = 0.$$
 \end{enumerate}
 \end{lemma}
 \begin{proof}
 Since $\mathcal R ^{(3)} $ is a  quotient of $W^k (sl(4),\theta)$, the   first assertion follows from Proposition \ref{omotacka}. Let us prove the second assertion.
We  notice that 
\bea  &&  :E^1 E^2 :  =  \nonumber \\   && ( S_{2} (\alpha + \beta - 3 \delta  ) + 
 6 S_2 (\tfrac{3}{2} (\delta + \varphi) )  +  \frac{9}{2}  (\alpha + \beta - 3 \delta) _{(-1)}  (\delta + \varphi) _{(-1)}  ) e^{\alpha + \beta}. \nonumber \\
 && :e \omega : =  \nonumber  \\ &&  (  - (\alpha_{(-1)} + \beta_{(-1)}  ) \beta_{(-1)} + \beta_{(-2)} +  \frac{3}{4} (\delta_{(-1)} ^2 - 2 \delta_{(-2)} - \varphi_{(-1)}^2 )  ) e^{\alpha + \beta}. \nonumber 
 \eea 
 
 By direct calculation we get the following relation:
 
 \begin{align*}
  &:E^1 E^2 : + 3 :e\omega: =  \\
 & (e_{(-3)} + \tfrac{3}{2} (h_{(-1)} e_{(-2)} - h_{(-2)} e_{(-1)} )\! +\!\tfrac{9}{2} ((\varphi_{(-1)} ^2 + \varphi_{(-2)} ) e_{(-1)} + 
 \varphi_{(-1)} e_{(-2)}))\vac.
 \end{align*}
  We have
 \bea
 && E^1 \circ E^2 = ( E^1 _{(-1)} + E^1 _{(0)} ) E^2=  \nonumber  \\
 && - 3 e_{(-1)} \omega +  e_{(-3)} + \frac{3}{2} (h_{(-1)} e_{(-2)} - h_{(-2)} e_{(-1)} ) +\frac{9}{2} (\varphi_{(-1)} ^2 + \varphi_{(-2)} ) e_{(-1)} \nonumber \\ && + \frac{9}{2}  \varphi(-1) e(-2)  + 3 e_{(-2)} + 9 \varphi_{(-1)} e_{(-1)}. \nonumber 
 \eea
 This gives the following relation in the Zhu  algebra:
 $$  -3 [e] [\omega] - 2 [e] + \frac{9}{2} ([j] ^2 + [j] -[j]) [e] = -3 [e] [\omega] - 2 [e] + \frac{9}{2} [j] ^2  [e] = 0. $$
 The claim follows.
 \end{proof}

 \begin{proposition} \label{simplicity}
\item[(1)] $ \mathcal R ^{(3)} $ is a simple vertex algebra.
\item[(2)] $ \mathcal R ^{(3)} \cong  (M\otimes \Pi(0) ) ^{int}$.
 \end{proposition}
 \begin{proof}
  
  By using the fact that $ \mathcal R ^{(3)} $ is a subalgebra of $M\otimes \Pi(0)$, we conclude that if   $ w_{sing}$ is a singular vector  for $W^k (sl(4),\theta)$ in $ \mathcal R ^{(3)} $, it must have $gl(2)$-weight $(n\omega_1, m)$  for $n \in {\Z_{\ge 0}}$ and $m \in \Z$.  This means that 
  $$ h_{(0)} w_{sing} = n w_{sing}, \quad \varphi_{(0)} w_{sing}  =  m w_{sing}.$$
  This leads to  the relation
  $$ L_{(0)} w_{sing} = ( \frac{ 3 n (n+2)  } {4} - \frac{3}{4} m^2) w_{sing}.$$
  On the other hand, $w_{sing}$ generates a submodule whose  lowest component must be  a module the for Zhu  algebra.   Now Lemma \ref{Zhus-relation} implies that
   $U(gl(2)) w_{sing} $ is annihilated by $[e] ([\omega] + \frac{2}{3} - \frac{3}{2}  [j] ^2  )$. If $n >0$, we get
   $$ \frac{ 3 n (n+2)  } {4} - \frac{3}{4} m^2 -\frac{3}{2}  m ^2  =  \frac{ 3 n (n+2) - 9 m^2 }{4}  =-\frac{2}{3}, $$
   which gives a contradiction since $m \in \Z$. So $n=0$. Then the fact that conformal weight must be positive implies that $m=0$. Therefore $w_{sing}$ must be proportional to the vacuum vector. We deduce that  there are no non-trivial singular vectors, and therefore $ \mathcal R ^{(3)} $ is a simple vertex algebra. This proves (1). The proof of assertion (2) is completely analogous.
 \end{proof}
 
\begin{proof}[Proof of  Theorem \ref{proof-conjecture}] Apply  Propositions  \ref{realization-r3} and  \ref{simplicity}.\end{proof}

\subsection {$\widehat{gl(2)}$-singular vectors in $\mathcal{R}^{(3)}$ }  

\begin{lemma} \label{sing-constr}Let $\ell \in {\Z}$.
\begin{enumerate}
\item If $\ell \ge 0 $,  then  for every $j \ge 0$ $$v_{\ell,j}= Q^j e ^{\frac{3 \ell}{2} (\delta + \varphi)  + 3  j \delta}$$ is a non-trivial singular vector in $\mathcal{R}^{(3)}$ .
\item If $\ell \le 0$, then for every $j \ge 0$  $$v_{\ell,j} =Q^{j-\ell} e ^{- \frac{3 \ell}{2} (\delta - \varphi)  + 3  j \delta}$$ is a non-trivial singular vector in $\mathcal{R}^{(3)}$.
\end{enumerate}
In particular, the set $\{ v_{\ell, j}\mid j \ge 0\}$ provides an infinite family of linearly independent  $\widehat{gl(2)}$-singular vectors in the $\ell$-eigenspace of $\varphi_{(0)}$.
\end{lemma}
\begin{proof}
The non-triviality of the singular vectors $v_{\ell.j}$ is well known (cf. \cite{AM1}). The assertions now follow from the fact $v_{\ell,j}$ belongs to a maximal $sl(2)$-integral part of $ M \otimes \Pi(0)$.\end{proof}

\subsection{Proof of Theorem \ref{A3special} }
Since we have proved that $W_k(sl(4),\theta)$ is isomorphic to the simple vertex algebra $\mathcal R^{(3)}$,  Lemma \ref{sing-constr} shows that each $ W_k(sl(4),\theta) ^{(i)}$ contains infinitely many linearly independent singular vectors.

\begin{rem}
Assertion (3) of Theorem \ref{proof-conjecture} implies that $ W_k(sl(4),\theta) = \mathcal{R}^{(3)}$ is an object of  the category $KL_{k+1}$ of $V_{k+1}(sl(2))$-module.  In particular, each $W_k (sl(4), \theta) ^{(i)}$ is an object  this category. Since  $k+1 =-5/3$ is a  generic level for $\widehat{sl(2)}$,  and the category $KL_{k+1}$ is semisimple (this follows easily from \cite{KL}, we skip details), we have that  $W_k (sl(4), \theta) ^{(i)}$ is completely reducible. So we actually proved that each  $W_k (sl(4), \theta) ^{(i)} $ is a direct sum of infinitely many irreducible $V_{k+1}(gl(2))$-modules.
\end{rem}

\section{Explicit decompositions from Theorem  \ref{finite-dec}: $\g^{\natural}$ is a Lie algebra}
\label{explicit-decomposition} 
In  Theorem  \ref{finite-dec} we proved a semisimplicity result for conformal embeddings of $\mathcal V_k(\g^{\natural})$ in $W_k(\g, \theta)$ where $\g=sl(n)$ or $\g=sl(2\vert n)$. But this semisimplicity result does not identify highest weights of the  components $W_k(\g, \theta)^{(i)}$. In this section we shall identify these components in certain cases and prove that then $W_k(\g, \theta)$ is a simple current extension of $\mathcal V_k(\g^{\natural})$.

Recall from Section \ref{Rankonelattice} that $F_n$ denotes a rank one lattice vertex algebra and $F^i_n$, $i=0,\cdots,n-1,$ denote  its irreducible modules.
The following result refines Theorem  \ref{finite-dec}.
\begin{theorem} \label{finite-dec-refined} \ 
(1)   If $\g=sl(2n)$ and $k = \frac{1}{2}-n$, $n \ge 2$, then 
$$\mbox{\rm Com} (V_{k+1} (sl(2n-2) ), {W}_{k}(\g, \theta) ) \cong F_{4 n (n-1)}.$$
 Moreover, we have the following decomposition of ${W}_{k}(\g, \theta)$ as a \break $V_{k+1} (sl(2n-2) ) \otimes F_{4 n (n-1)}$-module:
\bea {W}_{k}(\g, \theta) \cong \bigoplus_{i=0} ^{2n-3} L_{sl(2n-2) }  (k\Lambda_0 + \Lambda_i) \otimes F_{4 n (n-1) } ^{ 2 i n}. \label{dec-lattice-affine-1} \eea
(2)  If $\g=sl(2 \vert n)$ and $k = -\frac{2}{3} h ^{\vee} \notin \Z$, then $$\mbox{\rm Com} (V_{-k-1} (sl(n)), {W}_{k}(\g, \theta) ) \cong F_{ 3 n }.
$$ 
Moreover,  we have the following decomposition of ${W}_{k}(\g, \theta)$ as a \break $V_{-k-1} (sl(n)) \otimes F_{ 3 n }$-module:
\bea  {W}_{k}(\g, \theta) \cong \bigoplus_{i=0} ^{n-1} L_{sl(n)}  (-(k+2) \Lambda_0 + \Lambda_i) \otimes F_{ 3 n  } ^{ 3  i  }.  \label{dec-lattice-affine-2} \eea

\end{theorem}
\begin{proof}
(1) Let $\alpha = 2 n J ^{ \{ c\} }$ and note that 
$$\mathcal V_k (\g^{\natural})  = V_{k+1} (sl(2n-2)) \otimes M_{\alpha} ( 4 n (n-1) ).$$
 By  Theorem  \ref{finite-dec} we have that each  ${W}_{k}(\g, \theta)^{(i)}$ is an irreducible   $\mathcal V_k (\g^{\natural})$-module, and, by checking the action of $J^{\{c\}}_{(0)}$, we see that there is a weight $\L$ such that
 $$
 {W}_{k}(\g, \theta)^{(i)}=L_{sl(2n-2)}(\Lambda)\otimes M_{\a}(4 n (n-1),2in).
 $$

Since 
\bea {W}_{k}(\g, \theta)^{(1)} &\cong&   L_{sl(2n-2) }  (k\Lambda_0 + \Lambda_1) \otimes  M_{\alpha} ( 4n (n-1) , 2n),  \nonumber \\  {W}_{k}(\g, \theta)^{(-1)}  &\cong&
   L_{sl(2n-2) }  (k\Lambda_0 + \Lambda_{2n-3}) \otimes  M_{\alpha} (4 n (n-1) , - 2n), \nonumber \eea
the fusion rules result from Proposition \ref{f-r-affine-1}  and the  fusion rules   (\ref{fusion-heisenberg}) imply that 
\bea {W}_{k}(\g, \theta)^{(i)} \cong   L_{sl(2n-2) }  (k\Lambda_0 + \Lambda_ {\bar i} ) \otimes  M_{\alpha} ( 4 n (n-1) , 2 i n)  \label{for-11}, \eea
where
$$ \bar{i} \in \{ 0, \dots, 2n-3\}, \quad  i \equiv \bar{i} \ \mbox{mod} \ (2n-2). $$ 
Since 
$$ \mbox{Com} (V_{k+1} (sl(2n-2) ), {W}_{k}(\g, \theta) ) = \{ v \in {W}_{k}(\g, \theta) \vert \  J^{\{ u \} } _{(n)} v =  0, n \ge 0, u \in \g^{\natural} \},$$
we get that 
$$ \mbox{Com} (V_{k+1} (sl(2n-2) ), {W}_{k}(\g, \theta) )   \cong \bigoplus _{i \in \Z} M_{\alpha} ( 4n (n-1) , 4 i  n (n-1) )$$
as a $M_{\alpha} (4n (n-1) )$-module.
Now Proposition \ref{lattice-characterization} implies that 
$$ \mbox{Com} (V_{k+1} (sl(2n-2) ), {W}_{k}(\g, \theta) )   \cong F_{4n (n-1)} .$$
The decomposition (\ref{dec-lattice-affine-1})  now easily follows from (\ref{for-11}).
This proves (1). 

The proof of (2)  is based  on Proposition \ref{f-r-affine-2}  and it is completely analogous to the proof of assertion (1).
\end{proof}

\begin{rem}
Decompositions (\ref{dec-lattice-affine-1}) and  (\ref{dec-lattice-affine-2}), together with the  fusion rules result  from Propositions \ref{f-r-affine-1} and \ref{f-r-affine-2} imply that the minimal $W$-algebras from Theorem \ref{finite-dec-refined} are finite simple current extensions of the  tensor product of an admissible affine vertex algebra with a rank one lattice vertex algebra. It is also interesting to notice that the r.h.s. of (\ref{dec-lattice-affine-1}) and  (\ref{dec-lattice-affine-2}) have sense for the cases $\g = sl(n)$, $n$ odd,  and $\g = sl(2 \vert n)$, $k=-\frac{2}{3} h^{\vee} \in {\Z}$. But Corollary \ref{sl5} shows that most likely we won't get lattice vertex subalgebra in these cases. 
\end{rem}

\begin{rem}  The computation of the explicit decompositions in Theorem  \ref{finite-dec} when $\mathcal V_k(\g)$  does not contain  an admissible vertex algebra of type $A$  needs a subtler analysis. Our approach motivates the study of the following non-admissible affine  vertex algebras:
\begin{itemize}
\item $V_{k'} (sl(2n+1))$  for $k' = -n$, 
\item $V_{k'} (sl(n))$ for $k' =-\frac{2n+1}{3}$,
\item  $V_{k'} (sl(3 n+ 2))$  for $k' = -2 n -1$,
\item $V_{k'} (sl(n))$  for $k' = -\frac{n+1}{2}$, $n \ge 4$.
   \end{itemize} 
\end{rem}
\noindent Their representation theory is known only for $V_{-1}(sl(3))$ (cf. \cite{AP2}):
  \begin{proposition} \label{fr-ap}\cite{AP2}
  For $s \in \Z_{\ge 0}$ set
  $$U _s = L_{sl(3)} (-(1+s)  \Lambda _0 + s \Lambda_1), \quad U_{-s} = L_{sl(3)} (-(1+s)  \Lambda _0 + s \Lambda_2). $$
  \begin{itemize}
  \item The set $\{ U _s \ \vert \ s \in {\Z} \}$ provides a complete list of irreducible $V_{-1} (sl(3))$ modules from the category $KL_{-1}$.
  \item The following fusion rules hold in the category $KL_{-1}$:
  $$ U_{s_1} \times U _{s_2} =  U _{s_1 + s_2} \qquad (s_1, s_2 \in \Z). $$
  \end{itemize}
  \end{proposition}
  
  By using this proposition we get the following  refinement  of  Theorem  \ref{finite-dec} (1) for the case $n=5$:
  \begin{cor} \label{sl5} We have the following isomorphism of $V_{-2, 3/5}(\g^\natural)$-modules:
  $$ W_{-2} (sl(5), \theta) \cong \bigoplus _{s \in \Z} U_s \otimes  M (3/5 , s). $$
  %
 \end{cor}
  \begin{proof}Set $\a=J ^{ \{ c\} }$. Then $\mathcal V_{-2} (\g^{\natural})  = V_{-1} (sl(3)) \otimes M_{\alpha} ( 3/5 )$.
 By  Theorem  \ref{finite-dec}, we have that each  ${W}_{k}(\g, \theta)^{(i)}$ is an irreducible   $\mathcal V_k (\g^{\natural})$-module, and, by checking the action of $J^{\{c\}}_{(0)}$, we see that there is a weight $\L$ such that
 $$
 {W}_{k}(\g, \theta)^{(i)}=L_{sl(3)}(\Lambda)\otimes M_{\a}(3/5,i).
 $$
The assertion follows as in the proof of Theorem \ref{finite-dec-refined} from the fusion rules result from Proposition \ref{fr-ap}.
  \end{proof}
  
  \begin{rem}
  In \cite{KacV}, the vertex algebra $U _0$ and its modules $U_s$  from Proposition \ref{fr-ap} are realized inside of the Weyl vertex algebra  $M_3$ of rank three. It was proved in \cite{AP2} that
  $$ M_3  \cong \bigoplus _{s \in \Z} U _s \otimes  M (-3, s ).$$
 %
  Note that although $W_{-2} (sl(5), \theta)$ admits an analogous decomposition, one can easily see that this $W$-algebra is not  isomorphic to any subalgebra of $M_3$.
  \end{rem}

 \noindent  We also believe that the modules which appear in the decomposition of $W_k(\g, \theta)$ in Theorem  \ref{finite-dec} (3)  are also simple currents, so one can also expect the decomposition like in Corollary \ref{sl5}. Indeed, we can show that such decomposition holds but instead of applying  fusion rules  (which we don't know yet), we will apply  results from our previous papers \cite{AKMPP} and \cite{AKMPP1}.  In \cite{AKMPP} we proved  that the affine vertex algebra $V_{-\tfrac{n+1}{2}} (sl(n+1))$ ($n \ge 4$) is  semisimple as a $V_{-\tfrac{n+1}{2}} (gl(n))$-module and identified highest weights of all modules appearing in the decomposition. In \cite{AKMPP1} we proved that  $V_{-\tfrac{n+1}{2}} (sl(n+1))$  ($n \ge 4, n\ne 5$)  is embedded in the tensor product vertex algebra $W_k(sl(2\vert n), \theta) \otimes F_{-1}$. An application of these results will give the branching rules.
   
    For $s \in \Z_{\ge 0}$, we set
  $$U^{(n)}  _s = L_{sl(n)} (-(\tfrac{n+1}{2} +s)  \Lambda _0 + s \Lambda_1), \quad U^{(n)} _{-s} = L_{sl(n)} (-(\tfrac{n+1}{2} +s)  \Lambda _0 + s \Lambda_n). $$
   \begin{theorem} \label{sl(2,n)-weights}  Let $\g = sl(2 \vert n)$, $k =\frac{1-h^{\vee}} {2}$, $n = 4$ or $n \ge 6$. We have an isomorphism as
   $V_k(\g^\natural)$-modules:
  $$ W_{ k} (\g, \theta) \cong \bigoplus _{s \in \Z} U^{(n) } _s \otimes  M (\frac{n}{n-2} , s). $$
  %
 \end{theorem}
 \begin{proof}  We first consider the Heisenberg vertex algebra $M_{\alpha} (\frac{n}{n-2}) \otimes M_{\varphi} (-1)$ generated by the Heisenberg fields
 $\alpha = J^{\{ c\} } $ and $\varphi$ such that
 $$ [\alpha_{\lambda}\alpha]  = \frac{n}{n-2}  \lambda, \quad  [ \varphi _{\lambda}  \varphi  ] = - \lambda. $$ 
 Define 
$$ \overline  \varphi = \alpha + \varphi, \quad  \widehat \varphi  = \frac{2-n}{2} (\alpha + \frac{n}{n-2} \varphi  ).  $$ 
Then
$$ M_{\alpha} (\frac{n}{n-2}) \otimes M_{\varphi} (-1) = M _{\widehat \varphi  } (-\frac{n}{2})   \otimes M_{\overline \varphi} ( \frac{2}{n-2} ).  $$
 Theorem  \ref{finite-dec} (3)   implies that  
 $$ W_{ k} (\g, \theta)^{(s)}  \cong  L_{sl(n)} ( \Lambda ^{(s)} )  \otimes  M _{\alpha} (\frac{n}{n-2} , s), $$
 where $ L_{sl(n)} ( \Lambda ^{(s)} )  $ is an irreducible highest weight module from $KL_{-  \tfrac{n+1}{2} }$.
 It was proved in \cite[Theorem 5.6]{AKMPP1} that 
 \bea  \label{prva}   && V_{-\tfrac{n+1}{2}} (sl(n+1))  \otimes M_{\overline \varphi} (\frac{2}{n-2}) \nonumber \\ = &&  \bigoplus _{s \in \Z}  W_{ k} (\g, \theta)^{(s)} \otimes M_{\varphi} (-1, - s)  \nonumber \\ 
= && \bigoplus _{s \in \Z}   L_{sl(n)} ( \Lambda ^{(s)} )  \ \otimes M _{\alpha} (\frac{n}{n-2} , s) \otimes M_{\varphi} (-1, - s).     \nonumber \\ 
= && \left( \bigoplus _{s \in \Z}   L_{sl(n)} ( \Lambda ^{(s)} )  \ \otimes M _{\widehat \varphi  } (-\frac{n}{2}, s) \right)  \otimes M_{\overline \varphi} ( \frac{2}{n-2} ).   \nonumber 
\eea
 
 This implies that 
 
 \bea \label{druga} V_{-\tfrac{n+1}{2}} (sl(n+1))  \cong \bigoplus_{ s \in \Z}  L_{sl(n)} ( \Lambda ^{(s)} )      \otimes M_{\widehat \varphi } (-\frac{n}{2}, s). \nonumber  \eea
 Now results from \cite{AKMPP}  (see in particular \cite[Theorem 2.4, Theorem 5.1 (2)]{AKMPP}) imply that 
$$ L_{sl(n)} ( \Lambda ^{(s)} )   \cong U^{(n)}  _s. $$ 
 The claim follows.
 \end{proof}
   
\section{Explicit decompositions from Theorem  \ref{finite-dec}: $\g^{\natural}$ is not  a Lie algebra}\label{notlie}

  
   
 In this section we describe the decomposition of $W_{-2}(sl (n+5|n),\theta)$ as $\mathcal V_k (\g^\natural)$-module. We obtain, similarly to the results   of Section \ref{explicit-decomposition}, 
 that \break $W_{-2}(sl (n+5|n),\theta)$ is a simple current extension of $\mathcal V_k (\g^\natural)$. We expect this to hold in general.
 
\subsection{Simple current $V_{k_0} (sl(m \vert n) ) $--modules}
   
   Let us first recall  a few details on   
   simple current modules obtained by using the simple current operator $\Delta(\alpha,z). $\par Let $V$ be a conformal vertex algebra with conformal vector $\omega$. Let $\alpha$ be an even vector in $V$ such that
   $$ \omega_{n} \alpha = \delta_{n,0} \alpha, \quad \alpha_{(n)} \alpha = \delta_{n,1} \gamma {\bf 1} \qquad (n \ge 0),$$
   where $\gamma$ is a complex number.  Assume that $\alpha _{(0)}$ acts semisimply on $V$ with eigenvalues in $\Z$. Let  \cite{DLM}
   $$ \Delta (\alpha, z) = z^{\alpha (0)} \exp \left( \sum_{n =1} ^{\infty} \frac{\alpha_{(n)} }{-n} (-z)  ^{-n} \right). $$
   Then \cite{DLM}
  \bea (V^{(\alpha)}, Y_{\alpha} (\cdot, z) ) := (V, Y (\Delta(\alpha,z) \cdot, z) )\label{def-sc-1} \eea
   is a $V$--module, called a simple current $V$--module,  and 
   \bea Y_{\alpha} (\omega, z) 
   = Y (\omega, z) +  z^{-1} Y(\alpha,z) + 1/2 \gamma z ^{-2}. \label{vir-def} \eea
   When $V$ is the simple affine vertex algebra associated to  $sl(m|n)$, we will use   this construction  to produce  simple current modules  in a suitable category.\par    Let $\g = sl(m \vert n)$ $(m \ne n)$,  $k_0 \in {\C}$.  Let $e_{i,j}$ denote the standard matrix units  in $sl(m \vert n)$; consider the following vector in $ V_{k_0} (sl(m \vert n) )$:
$$\alpha^{m,n}  = \frac{1}{m-n}  ( n  e_{1,1} + \cdots + n  e_{m, m} + m  e_{m+1,m+1} + \cdots + m e_{m+n, m+n}  )_{(-1)}{\bf 1}. $$ 
Note that \begin{equation}\label{decomp} \g = \g^ {-1} + \g^ 0 + \g ^1, \quad \g^ {i } = \{ x \in \g \vert \ [\alpha^{m,n}, x ] = i x\}. \end{equation}
In particular, $ \g^ 0 \cong  sl(m) \times sl(n) \times {\C} \alpha^{m,n} $ is the even part  of $\g$ and ${\g}^{-1} + {\g}^{1}$ is its odd part. 

For $i \in \{-1,0,1\} $ and $n \in \Z$ set $$\g ^{i} (n) = \g^{i} \otimes t ^n.$$
The  decomposition \eqref{decomp} implies that  $\alpha^{m,n} _{(0)}$ acts semi-simply on  $V_{k_0}  (sl(m \vert n))$ with integral eigenvalues. Moreover
$$[{\alpha^{m,n}}_{\ \lambda}\ \alpha ^{m,n} ]  = \frac{1} {(m-n) ^2} ( n ^ 2 m k_0 - m ^2 n  k_0) \lambda = - \frac{ n  m k_0} {m-n} \lambda. $$ Set
   $$ U^{m,n} _s = V_{k_0}  (\g ) ^{(s \alpha^{m,n} )} \quad (s \in \Z). $$
   By definition (\ref{def-sc-1}), we see that $ U^{m,n} _s$ is obtained from  $ V_{k_0}  (\g )$ by applying the  automorphism $\pi_s$   of $\widehat{\g}$ (and $V_{k_0}  (\g )$)   uniquely determined by
   \bea  \pi_s ( x^{\pm }_{(r)} )  & = &  x^{\pm}_{(r \mp s)} \quad  (x^{\pm}  \in \g ^{\pm 1} ), \label{aut-1}\\
   \pi_s ( x_{(r)}) & =& x_{(r)}  \quad  (x   \in  sl(m) \times sl(n) \subset {\g} ^{0 }  ), \label{aut-2} \\
   \pi_s ( \alpha^{m,n}_{(r)})  &=&   \alpha^{m,n}_{(r)}    - \frac{ n  m k_0} {m-n}   s  \delta_{r,0}, \label{aut-3} 
    \eea
    where $r, s \in \Z$.
Note that in $U^{m,n} _s$ we have $$ {\g}^{\pm 1}(n \pm s).{\bf 1} =0 \quad (n \ge 0). $$
   
   \begin{theorem}  \label{sc-1} Assume that $m , n \ge 1$.  We have:\par\noindent 
(1) $U^{m,n} _s$, $s \in \Z$,  are  irreducible   $V_{k_0}  (\g ) $--modules from the category $KL_{k_0}$.\par\noindent
(2) Let $s =\pm 1$. Then the lowest  graded component of $U^{m,n} _s$ is,  as a vector space,  isomorphic to 
$$ \bigwedge \left( {\g} ^{ s } (0) \right) . {\bf 1}. $$
It has conformal weight   $- \frac{ n  m k_0} {m-n} $. 
\par\noindent (3)  $U^{m,n} _s$,  $s \in \Z$,  are   simple current $V_{k_0} (\g)$--modules in $KL_{k_0}$ and  the following fusion rules holds in $KL_{k_0}$:
\bea  &&U^{m,n} _{s_1} \times U^{m,n} _{s_2} =  U ^{m,n} _{s_1 + s_2} \qquad (s_1, s_2 \in \Z). \label{f-r} \eea
\end{theorem}
\begin{proof}
(1) Since $U^{m,n} _s = \pi_s ( V_{k_0} (\g))$,  we get that $U^{m,n} _s$ is irreducible $V_{k_0}  (\g )$--module. Relations (\ref{aut-1}) -(\ref{aut-3}) together with  (\ref{vir-def}) imply that $U^{m,n} _s$ belongs to $KL_{k_0}$. In fact, the  lowest  graded component is contained in  the vector space
\bea &&  \bigwedge \left(\g ^1 (0) + \cdots +\g^1 (s-1) \right) . {\bf 1} \quad (s \ge 1) \nonumber \\
&&  \bigwedge \left(\g ^{-1} (0) + \cdots  +\g^{-1}  (-s -1) \right). {\bf 1} \quad (s \le -1) .\nonumber \eea
For $s=\pm 1$ we get assertion (2).
Assertion (3) follows from \cite{Li-ext}. More precisely,  for  any irreducible $V_{k_0}  (\g ) $--module $(M,Y_M)$  from the category $KL_{k_0}$ one can show that 
 \bea (M_s  , \widetilde Y_{M} (\cdot, z) ) = (M, Y_M (\Delta(\alpha^{m,n} ,z) \cdot, z) )\label{def-sc-2} \eea
is also an irreducible $V_{k_0}  (\g ) $--modules from the category $KL_{k_0}$ (this follows from the fact that  $M_s $  is essentially obtained by applying the automorphism $\pi_s$). Then \cite[Theorem  2.13]{Li-ext} gives the fusion rules
$$ M \times U^{m,n} _{s}   =  M _s  .  $$
In particular,  this proves the fusion rules (\ref{f-r}). 
\end{proof}







\begin{rem}  Let us consider the case $s=\pm 1$. Then  the lowest  weight component $U^{m,n}  _s (0)$ of $U^{m,n} _s$ is  a irreducible (sub)quotient of the Kac module $K^s _{m,n} (k_0)$ induced from the $1$--dimensional $(\g^0+ \g^{-s})$--module  ${\C}  {\bf 1}$ with action
\bea
\g^{-s} . {\bf 1} & =& 0, \nonumber \\
 x. {\bf 1} &=& 0 \qquad (x \in sl(m) \times sl(n)), \nonumber \\
 \alpha^{m,n} . {\bf 1} &=&  - s \frac{ n  m k_0} {m-n} {\bf 1}. \nonumber 
\eea
As a vector space $K^s _{m,n} (k_0 ) \cong \bigwedge \g^ s$.
If we take  an odd coroot $\beta = e_{i,i} + e_{m+j, m+j}$, $1 \le i \le m$, $ 1 \le j \le n$,  by direct calculation we get
\begin{equation}\label{beta}\beta . {\bf 1}  = -  s k_0. {\bf 1}. \end{equation}
This implies that  $K^ s _{m,n}(k_0)$ is 
 typical iff  $k_0\notin\{-(m-1),\ldots,n-1\}$.
 
Recall that  a weight $\l$ of a basic Lie superalgebra is said to be   typical if  $(\l+\rho)(\beta)\ne 0$ for each isotropic odd root $\beta$. To derive the above condition on $k_0$, we make computations in a distinguished set of positive roots;  we have 
 $$\rho=-\frac{s}{2}\left(\sum_{i=1}^m(m-n-2i+1)\e_i+\sum_{j=1}^n(n+m-2j+1)\d_j\right),$$
and from \eqref{beta} we deduce that 
 $$(\l+\rho)(\beta)=-s(k_0+m-i-j+1), $$
which is non-zero  if $k_0\notin\{-(m-1),\ldots,n-1\}$.
Under this hypothesis, the lowest graded component $U^{m,n}_ s (0)$ of  $U^{m,n}_ s$ is isomorphic to $K^s _{m,n} (k_0)$ as a  $\g$--module.
\end{rem}
Now we specialize the previous construction  to $\g = sl(4\vert 1)$ and $k_0 = -1$. So$$\alpha = \alpha ^{4,1}  = \frac{1}{3}  ( e_{1,1} + \cdots + e_{4, 4} + 4 e_{5,5} )_{(-1)} {\bf 1} \in V_{-1} (sl(4 \vert 1)).$$ and 
$$ [\alpha _{\lambda} \alpha] =  \frac{4}{3} \lambda, \qquad U_s = U^{4,1} _s =V_{-1}  (sl(4  \vert 1)) ^{(s \alpha)} \quad (s \in \Z). $$
\begin{cor} \label{c-4.1} We have:
\par\noindent (1) $U_s$, $s \in \Z$, are  irreducible $V_{-1}  (sl(4   \vert 1)) $--modules from the category $KL_{-1}$.
\par\noindent (2) The lowest graded component of  $U_{1}$ is isomorphic to ${\C} ^{ 4  \vert 1}$ and that of $U_{-1}$ is isomorphic to $({\C} ^{ 4 \vert 1})^*$.
\par\noindent (3)  $U_s$ is an irreducible, simple current $V_{-1}  (sl(4  \vert 1))$--module  and  the following fusion rules hold:
$$ U_{s_1} \times U _{s_2} =  U _{s_1 + s_2} \qquad (s_1, s_2 \in \Z). $$
\end{cor}
\begin{proof}
Proof follows from Theorem \ref{sc-1} and the fact that top component of $U_1$ (resp. $U_{-1}$) has the same highest weight as the $sl(4 \vert 1)$--module ${\C} ^{4 \vert 1}$ (resp.  $({\C} ^{4 \vert 1} )^{*}$.
\end{proof}

Modules $U_s$  are actually  obtained from  the vertex algebra  $V_{k_0} (sl(4 \vert 1))$ by applying the spectral flow automorphism $\pi_s$  of $\widehat{sl(4 \vert 1)}$ which leaves $\widehat{sl(4)}$--invariant.

\subsection{The decomposition for $W_k( sl(6 \vert 1),\theta)$, $k =-2$.}

We now consider the minimal $W$--algebra $W_k( \g ,\theta)$ for Lie superalgebra $\g = sl( 6 \vert 1)$ at the  conformal, non-collapsing level $k=-2$.    We shall prove that each $ W_k( \g,\theta )^{ (i)} $ is a simple current $\mathcal V_k (\g^{\natural})$--module. In order to see this, essentially  it suffices to prove that $ W_k( \g,\theta )^{ (\pm 1 )} $ are the simple current modules described in previous section.
Note that $\g ^{\natural} = sl(4 \vert 1) + {\C}$, and that $$\mathcal V_k (\g) = V_{k+1} (sl(4 \vert 1) ) \otimes M_{\beta} ( \tfrac{2 h^{\vee} -4} { h^{\vee}  } (k + h ^{\vee} /2 )),$$
where $\beta = J^{\{ c\}} $, $[ \beta _{\lambda} \beta ]  =  \frac{2 h^{\vee} -4} { h^{\vee}  } (k + h ^{\vee} /2 ) = \frac{3}{5}.$

By the irreducibility statement  from  Theorem  \ref{finite-dec} we see that there are weights $\Lambda ^{\pm}$ such that
$$ W_k( \g ,\theta)^{ (\pm 1 )} \cong L_{sl(4 \vert 1)}(\Lambda^{\pm} ) \otimes M_{\beta} ( \tfrac{3 } { 5 }, \pm 1  ). $$ 
 The lowest graded  component of $L_{sl(4 \vert 1)}(\Lambda^{+} ) $ (resp.  $L_{sl(4 \vert 1)}(\Lambda^{+} ) $ ) is  isomorphic as $sl(4 \vert 1)$--module  to ${\C} ^{ 4 \vert 1}$  (resp. $( {\C} ^{ 4 \vert 1} ) ^{*}$ ) and it has 
conformal weight  $h_1 = \frac{2}{ 3 }$.
By  Corollary \ref{c-4.1}, we get that 
$$ W_k( \g,\theta )^{ (\pm 1 )} \cong U_{\pm 1} \otimes M_{\beta} ( \tfrac{3 } { 5 }, \pm 1  ). $$ 
Since $U_{\pm 1}$ and $M_{\beta} ( \tfrac{3 } { 5 }, \pm 1  )$ are simple current modules we get that $W_k(\g,\theta)$ is a simple current extension. In this way we have proved the following result, which gives a super-analog of Corollary \ref{sl5}. (Arguments are essentially the same, only the proof that $W_k( \g,\theta )^{ (\pm 1 )} $ are simple current modules uses different techniques).

  \begin{cor} \label{sl6,1}  Let $\g = sl(6\vert 1)$. We have the following isomorphism of $V_{-2, 3/5}(\g^\natural)$-modules:
  $$ W_{-2} (\g , \theta) \cong \bigoplus _{s \in \Z} U_s \otimes  M_{\beta}  (3/5 , s). $$
  %
 \end{cor}

  \begin{rem}\label{85}
  By using similar arguments  one can obtain analogous decompositions for $\g=sl (n+5|n)$ and conformal level $k=-2$. For decompositions in the case of other conformal levels we  need more precise fusion rules analysis. This and related questions will be discussed in our forthcoming papers.
  \end{rem}

  \footnotesize{
  \noindent{\bf D.A.}:  Department of Mathematics, Faculty of Science, University of Zagreb, Bijeni\v{c}ka 30, 10 000 Zagreb, Croatia;
{\tt adamovic@math.hr}
  
\noindent{\bf V.K.}: Department of Mathematics, MIT, 77
Mass. Ave, Cambridge, MA 02139;\newline
{\tt kac@math.mit.edu}

\noindent{\bf P.MF.}: Politecnico di Milano, Polo regionale di Como,
Via Valleggio 11, 22100 Como,
Italy; {\tt pierluigi.moseneder@polimi.it}

\noindent{\bf P.P.}: Dipartimento di Matematica, Sapienza Universit\`a di Roma, P.le A. Moro 2,
00185, Roma, Italy; {\tt papi@mat.uniroma1.it}

\noindent{\bf O.P.}:  Department of Mathematics, Faculty of Science, University of Zagreb, Bijeni\v{c}ka 30, 10 000 Zagreb, Croatia;
{\tt perse@math.hr}
}


\begin{thebibliography}{100}
\bibitem{Abe}  T. Abe, \emph{A $\Z _2$-orbifold model of the symplectic fermionic vertex operator superalgebra.} Math. Z. 255
(2007), 755--792.
\bibitem{A-1999}D. Adamovi\' c, \emph{Representations of the $N = 2$ superconformal vertex algebra},  IMRN \textbf{2} (1999), 61--79.
\bibitem{A-2014} D.~Adamovi\' c, \emph{A realization of certain modules for the $N=4$ superconformal algebra and the   affine Lie algebra $A_2 ^{(1) }$},  Transformation Groups, Vol. 21, No. 2  (2016) 299--327,  arXiv:1407.1527.
\bibitem{A-2016} D. Adamovi\' c, \emph{The vertex algebras $\mathcal R ^{(p)}$ and their logarithmic representations}, in preparation
\bibitem{AM1} D. Adamovi\'c, A. Milas,  \emph{Vertex operator algebras associated to  modular invariant representations for  $A_1 ^ {(1)}$} ,  Mathematical Research Letters \textbf{2} (1995) , 563-575.  
\bibitem{AM2} D. Adamovi\'c, A. Milas , \emph{On the triplet vertex algebra $W(p)$}, Adv.  Math. \textbf{217} (2008), 2664--2699.
\bibitem{A} D.~Adamovi\'c, O.~Per\v{s}e, \emph{Some General Results on Conformal Embeddings of Affine Vertex Operator Algebras}, Algebr. Represent. Theory \textbf{16} (2013), no. 1, 51--64.
\bibitem{AP2} D.~Adamovi\'c, O.~Per\v{s}e, \emph{Fusion Rules and Complete Reducibility of
Certain Modules for Affine Lie Algebras},  Journal of algebra and its applications \textbf{13}, n.1,  1350062 (2014), 18pp.
\bibitem{AKMPP} D.~Adamovi\'c, V.~G. Kac, P.~M\"oseneder Frajria, P.~Papi, O.~Per\v{s}e, \emph{Finite vs infinite decompositions in conformal embeddings}, arXiv:1509.06512, Communications in Mathematical Physics,
\textbf{348}, 445-473 (2016)
\bibitem{AKMPP1} D.~Adamovi\'c, V.~G. Kac, P.~M\"oseneder Frajria, P.~Papi, O.~Per\v{s}e, \emph{Conformal embeddings of affine vertex algebras  in minimal $W$--algebras I: structural results}, arXiv:1602.04687,
 to appear in Journal of Algebra, doi 10.1016/j.jalgebra.2016.12.005 

 

\bibitem{Ar-2005}  T. Arakawa, \emph{Representation theory of superconformal algebras and the Kac-Roan-Wakimoto conjecture}, Duke
Math. J. \textbf{130} (3) (2005) 435--478.

\bibitem {Ar-rationality} T. Arakawa, \emph{Rationality of admissible affine vertex algebras in the category $\mathcal O$}, Duke Math. J, Volume 165, Number 1 (2016), 67--93.


 

\bibitem{BK2} B. Bakalov, V.~G. Kac, \emph{Field algebras.}
Int. Math. Res. Not. 2003, no. 3, 123--159. 
\bibitem{BK} B. Bakalov, V.~G. Kac, \emph{Generalized vertex algebras}, in Lie theory and its applications in physics VI, ed. H.-D. Doebner and V.K. Dobrev, Heron Press, Sofia, 2006, 3--25.

\bibitem{TC} T. Creutzig, \emph{ $W$--algebras for Argyres-Douglas theories}, 	arXiv:1701.05926
\bibitem{Dong} C. Dong, \emph{Vertex algebras associated with even lattices}, J. Algebra 161 (1993),
245-265.


\bibitem{DK} A. De Sole, V. G.  Kac, \emph{Finite vs affine W -algebras}, Jpn. J. Math. 1 (2006), 137--261.

\bibitem{DL} C. Dong, J. Lepowsky, \emph{Generalized vertex algebras and relative vertex operators},
Progress in Math. 112, Birkhauser Boston, 1993.
 \bibitem{DLM}
C.~Dong,  H.~Li, G.~Mason, \emph{
Simple currents and extensions of vertex operator algebras.} Comm. Math. Phys., 180 (1996), pp. 671--707.

 \bibitem{DLM-reg}
C.~Dong,  H.~Li, G.~Mason, \emph{
Regularity of rational vertex operator algebras} , Adv. Math. 132 (1997), no. 1, 148--166.


\bibitem{FS} B. L. Feigin, A. M. Semikhatov  \emph{ The $\hat{sl}(2)+\hat{sl}(2)/\hat{sl}(2)$ Coset Theory as a Hamiltonian Reduction of $\hat{D}(2\vert 1;\alpha)$}, Nucl. Phys. B 610 (3), 489-530 (2001); hep-th/0102078.

   
 \bibitem{FZ} I.~Frenkel, Y.~Zhu, \emph{
Vertex operator algebras associated to representations of affine and Virasoro algebras. }
Duke Math. J. \textbf{66} (1992), no. 1, 123--168. 
   \bibitem{Lep} I.~Frenkel, Y.~Huang, J.~Lepowsky, \emph{On axiomatic approaches to vertex operator algebras and modules.} Mem. Amer. Math. Soc. \textbf{104} (1993), no. 494.
   
   \bibitem{GK}  M.~Gorelik, V.G.~Kac, \emph{On simplicity of vacuum modules}. Adv. Math., 211 (2007), 621-677.
   
   \bibitem{GKMP}  M.~Gorelik, V.G.~Kac, P.~M\"oseneder Frajria, P.~Papi, \emph{ Denominator identities for finite-dimensional Lie superalgebras and Howe duality for compact dual pairs}, Japanese Journal of Mathematics , 7 (2012),41--134
 \bibitem{KacWW}V.~G. Kac and M.~Wakimoto,
\emph{Classification of modular invariant representations of affine algebras}, in Infinite dimensional Lie algebras and groups, Adv. Ser. Math. Phys. 7, World Scientific (1989), 138--177.
 \bibitem{KacLie} V.~G. Kac, \emph{Lie superalgebras}, Adv. Math. \textbf{26}, (1977),  8--96.
 \bibitem{KacV} V.~G. Kac, \emph{Vertex algebras for beginners}, second ed., University Lecture
  Series, vol.~10, American Mathematical Society, Providence, RI, 1998.

\bibitem{KWR} V.~G. Kac, S.-S. Roan, M. Wakimoto, \emph{Quantum reduction for affine superalgebras}, Comm.
Math. Phys. \textbf{241} (2003) 307--342 math-ph/0302015.
 \bibitem{KW} V.~G. Kac, M.~Wakimoto,
\emph{Quantum reduction and representation theory of superconformal algebras}, Adv. Math., \textbf{185} (2004), pp. 400--458.
\bibitem{KW2} V.~G. Kac, M.~Wakimoto,
\emph{Corrigendum to: "Quantum reduction and representation theory of superconformal algebras''},  Adv. Math. \textbf{193} (2005), no. 2, 453--455.

\bibitem{KW-1988} V.~G. Kac, M.~Wakimoto, \emph{Modular invariant representations of infinite dimensional
 Lie algebras and superalgebras},  Proc. Natl. Acad. Sci. USA{\bf
 85} (1988), 4956--4960.
 \bibitem{KMP}V.~G. Kac, P.~M\"oseneder Frajria, P.~Papi, F.~Xu, \emph{
 Conformal embeddings and simple current extensions},  IMRN  (2015), no. 14, 5229--5288.

\bibitem{KL} D. Kazhdan and G. Lusztig, \emph{Tensor structures arising from affine Lie algebras, I, II}, J. Amer. Math. Soc. 6:4 (1993), 905--947, 949--1011. 

\bibitem{KWn} V.~G. Kac and W. Wang, \emph{Vertex operator superalgebras and their representations},
Contemp. Math. 175, (1994), 161--191.

\bibitem{Li-ext} H. Li, \emph{Certain extensions of vertex operator algebras of affine type},  Comm. Math. Phys. 217 (2001), no. 3, 653--696.
 \bibitem{LX} H. Li, X. Xu, \emph{A characterization of vertex algebras associated to even lattices}, J. Algebra 173 (1995) 253--270.
 
 
 
 \bibitem{Zhu}  Y. Zhu, \emph{Modular invariance of characters of vertex operator algebras},
  J. Amer. Math. Soc. {\bf  9} (1996), 237--302.
  
  \end{thebibliography}
\end{document}